\documentclass[12pt]{amsart}

\usepackage[text={425pt,650pt},centering]{geometry}
\usepackage{color}
\usepackage{esint,amssymb}
\usepackage{graphicx}
\usepackage{tikz}

\newtheorem{theorem}{Theorem}
\newtheorem{proposition}[theorem]{Proposition}
\newtheorem{lemma}[theorem]{Lemma}
\newtheorem{corollary}[theorem]{Corollary}

\theoremstyle{definition}
\newtheorem{remark}[theorem]{Remark}

\newcommand{\cref}[1]{Corollary~\ref{c.#1}}

\numberwithin{equation}{section}
\numberwithin{theorem}{section}

\newcommand{\N}{\mathbb{N}}
\newcommand{\R}{\mathbb{R}}

\newcommand{\HH}{\mathcal{H}}

\newcommand{\ep}{\varepsilon}

\newcommand{\average}{{\mathchoice {\kern1ex\vcenter{\hrule height.4pt
width 6pt depth0pt} \kern-9.7pt} {\kern1ex\vcenter{\hrule
height.4pt width 4.3pt depth0pt} \kern-7pt} {} {} }}
\newcommand{\ave}{\average\int}
\def\R{\mathbb{R}}

\renewcommand{\bar}{\overline}
\renewcommand{\tilde}{\widetilde}

\begin{document}

\title[On the free boundary of the obstacle problem]{On the fine structure of the free boundary\\ for the classical obstacle problem}

\begin{abstract}
In the classical obstacle problem,
the free boundary can be decomposed into ``regular'' and ``singular'' points. As shown by Caffarelli in his seminal papers \cite{C77,C98}, regular points consist of smooth hypersurfaces, while singular points are contained in a stratified union of $C^1$ manifolds of varying dimension. In two dimensions, this $C^1$ result has been improved to $C^{1,\alpha}$ by Weiss \cite{W99}.

In this paper we prove that, for $n=2$  singular points are locally contained in a $C^2$ curve. 
In higher dimension $n\ge 3$, we show that the same result holds with $C^{1,1}$ manifolds (or with countably many $C^2$ manifolds), up to the presence of some ``anomalous'' points of higher codimension.
In addition, we prove that the higher dimensional stratum is always contained   in a $C^{1,\alpha}$ manifold, thus extending to every dimension the result in \cite{W99}.

We note that, in terms of density decay estimates for the contact set, our result is optimal.
In addition, for $n\ge3$ we construct  examples of very symmetric solutions exhibiting linear spaces of anomalous points, proving that our bound on their Hausdorff dimension is sharp.
\end{abstract}

\author[A. Figalli]{Alessio Figalli}
\address{ETH Z\"urich, Mathematics Dept., R\"amistrasse 101, 8092 Z\"urich, Switzerland.}
\email{alessio.figalli@math.ethz.ch}

\author[J. Serra]{Joaquim Serra}
\address{ETH Z\"urich, Mathematics Dept., R\"amistrasse 101, 8092 Z\"urich, Switzerland.}
\email{joaquim.serra@math.ethz.ch}

\date{\today}

\maketitle

\section{Introduction}


The classical obstacle problem consists in studying the regularity of solutions to the minimization problem
$$
\min_{v}\biggl\{\int_{B_1}\frac{|\nabla v|^2}{2}\,:\, v \geq \psi \text{ in $B_1$},\,v|_{\partial B_1}=g\biggr\},
$$
where $g:\partial B_1\to \R$ is some prescribed boundary condition, and the ``obstacle'' $\psi:B_1\to \R$
satisfies $\psi|_{\partial B_1}<g$.

Assuming that $\psi$ is smooth, it is well-known that this problem has a unique solution $v$ of class $C^{1,1}_{\rm loc}$ \cite{BK74}, and that $u:=v-\psi$ satisfies the Euler-Lagrange equation
$$
\Delta u=-\Delta\psi \,\chi_{\{u>0\}} \qquad \text{in $B_1$}.
$$

As already observed in \cite{C77,C98}, in order to prove some regularity results for the \emph{free boundary} $\partial\{u>0\}$ it is necessary to assume that $\Delta\psi<0$.
In addition, as also noticed in \cite{C98,W99,M03,PSU12}, from the point of view of the local structure it suffices to understand the model case $\Delta\psi\equiv -1$.
For this reason, from now on, we shall focus on the problem
\begin{equation}\label{obstaclepb}
\Delta u=\chi_{\{u>0\}},\quad u \geq 0 \qquad \mbox{in } B_1\subset \R^n.
\end{equation}

As shown by Caffarelli in his seminal papers \cite{C77,C98}, points of the {free boundary} $\partial\{u>0\}$ are divided into two classes: regular points and singular points. A free boundary point $x_\circ$ is either regular or singular depending on the type of blow-up of $u$ at that point. More precisely:
\begin{equation}\label{regular}
x_\circ \mbox{ is called \emph{regular} point}  \quad \Leftrightarrow \quad  \quad r ^{-2} u(x_\circ+ rx) \ \stackrel{r\downarrow 0}\longrightarrow\  \frac 1 2 \max\{\boldsymbol e\cdot x,0\}^2 
\end{equation}
for some $\boldsymbol e=\boldsymbol e_{x_\circ}\in \mathbb S^{n-1}$, and
\begin{equation}\label{singular}
x_\circ \mbox{ is called \emph{singular} point}  \quad \Leftrightarrow \quad  r ^{-2} u(x_\circ+ rx) \ \stackrel{r\downarrow 0}\longrightarrow \ p_{*,x_\circ}(x) :=\frac 1 2 x\cdot A x 
\end{equation}
for some symmetric nonnegative definite matrix $A=A_{x_\circ}$ with  ${\rm tr}(A)=1$. 
The existence of the previous limits in \eqref{regular} and \eqref{singular},
as well as the classification of possible blow-ups are well-known results; see \cite{C98,W99,M03,PSU12}.

By the theory in \cite{KN77,C77} (see also \cite{CR76,CR77,Sak91,Sak93,C98,M00,PSU12}), the free boundary is an analytic hypersurface near regular points. On the other hand, near singular points the \emph{contact set} $\{u=0\}$ forms cups and can be pretty wild ---see for instance the examples given in  \cite{Sch76} and  \cite{KN77}. Moreover, as shown in \cite{Sch76}, even $C^\infty$ strictly superhamonic obstacles in the plane ($n=2$) may lead to  contact sets with Cantor set like structures. In particular, in such examples, the contact set has (locally) an infinite number of connected components, each containing singular points. 

Despite these ``negative'' results showing that singular points could be rather bad, it is still possible to prove some nice structure. More precisely, singular points are naturally stratified according to the dimension of the linear space  
\[L_{x_\circ} := \{ p_{*,x_\circ} =0 \} = {\rm ker}(A_{x_\circ}).\]
For $m\in \{0,1,2,\dots, n-1\}$ we define the $m$-th stratum as
\[
\Sigma_{m} := \big\{ x_\circ \ : \ \mbox{ singular point with } \dim( L_{x_\circ}) =m\big\}.
\]
As shown by Caffarelli in \cite{C98},
each stratum $\Sigma_m$ is locally contained in a $m$-dimensional manifold of class $C^1$ (see also \cite{M03} for an alternative proof). 
This result has been improved in dimension $n=2$ by Weiss \cite{W99}: using a epiperimetric-type approach, he has been able to prove that $\Sigma_1$ is locally contained in a $C^{1,\alpha}$ curve, for some universal exponent $\alpha>0$.
Along the same lines, in a recent paper Colombo, Spolaor, and Velichkov  \cite{CSV17} have obtained  a logarithmic epiperimetric inequality at singular points in any dimension $n\ge3$, thus improving the known $C^{1}$ regularity to a more quantitative $C^{1,\log^{\epsilon}}$ one.
\\

The aim of this paper is to improve the previous known results by showing that, up to the presence of some ``anomalous'' points of higher codimension,  singular points can be covered by $C^{1,1}$ (and in some cases $C^2$) manifolds. As we shall discuss in Remark \ref{rmk:optimal}, this result provides the optimal decay estimate for the contact set. In addition, anomalous points may exist and our bound on their Hausdorff dimension is optimal.

Before stating our result we note that, as a consequence of \cite{C98}, points in $\Sigma_0$ are isolated and $u$ is strictly positive in a neighborhood of them. In particular $u$ solves $\Delta u=1$ in a neighborhood of $\Sigma_0$, hence it is analytic there. Thus, it is enough to understand the structure of $\Sigma_m$ for $m=1,\ldots,n-1$.

Here and in the sequel, ${\rm dim}_{\mathcal H}(E)$ denotes the Hausdorff dimension of a set $E$ (see \eqref{eq:def dim} for a definition).
Our main result is the following:
\begin{theorem}
\label{thm:main}
Let $u \in C^{1,1}(B_1)$ be a solution of \eqref{obstaclepb}, and let $\Sigma:=\cup_{m=0}^{n-1}\Sigma_m$ denote the set of singular points. Then:
\begin{enumerate}
\item[($n=2$)] $\Sigma_1$ is locally  contained in a $C^{2}$ curve.
\item[($n\ge 3$)] \begin{enumerate}
\item
The higher dimensional stratum $\Sigma_{n-1}$ 
can be written as the disjoint union of ``generic points'' $\Sigma_{n-1}^g$ and ``anomalous points'' $\Sigma_{n-1}^a$, where:\\
- $\Sigma^g_{n-1}$ is locally contained in a $C^{1,1}$ $(n-1)$-dimensional manifold;\\
- $\Sigma^a_{n-1}$ is a relatively  open subset of $\Sigma_{n-1}$ satisfying  ${\rm dim}_{\mathcal H}(\Sigma^a_{n-1})\leq n-3$  (actually, $\Sigma^a_{n-1}$ is discrete when $n=3$).

Furthermore, $\Sigma_{n-1}$ can be locally covered by a $C^{1,\alpha_\circ}$  $(n-1)$-dimensional manifold, for some dimensional exponent $\alpha_\circ>0$. 
\item For all $m=1,\ldots,n-2$ we can write $\Sigma_{m}=\Sigma_m^g\cup\Sigma_g^a$,
where:\\
- $\Sigma_m^g$ can be locally covered by a $C^{1,1}$
$m$-dimensional manifold;\\
- $\Sigma^a_m$ is a relatively open subset of $\Sigma_{m}$ satisfying ${\rm dim}_{\mathcal H}(\Sigma^a_m)\leq m-1$ (actually, $\Sigma^a_m$ is discrete when $m=1$).

In addition, $\Sigma_{m}$ can be locally covered by a $C^{1,\log^{\epsilon_\circ}}$ $m$-dimensional manifold, for some dimensional exponent $\epsilon_\circ>0$. 
\end{enumerate}
\end{enumerate}
\end{theorem}

\begin{remark}\label{rmk:optimal}
We first discuss the optimality of the above theorem.
\begin{enumerate}
\item
Our $C^{1,1}$ regularity provides the optimal control on the contact set in terms of the density decay. Indeed our result implies that, at all singular points up to a $(n-3)$-dimensional set (in particular at all singular points when $n=2$, and at all singular points up to a discrete set when $n=3$), the following bound holds:
$$
\frac{|\{u=0\}\cap B_r(x_\circ)|}{|B_r(x_\circ)|}\leq Cr\qquad \forall\,r>0
$$
(see Proposition \ref{decayest},  Definition \eqref{eq:def anomalous}, and Lemmas \ref{lem11}, \ref{lem12}, \ref{lem21}, and \ref{lem22}).
In view of the two dimensional Example 1 in \cite[Section 1]{Sch76}, this estimate is optimal.

\item The possible presence of anomalous points comes from different reasons depending on the dimension of the stratum. More precisely, as the reader will see from the proof (see also the description of the strategy of the proof given below), the following holds:
\begin{enumerate}
\item The possible presence of points in $\Sigma_{n-1}^a$ comes from the potential existence, in dimension $n\ge 3,$ of $\lambda$-homogeneous solutions to the Signorini problem with $\lambda\in(2,3)$. Whether this set is empty or not is an interesting open problem.
 \item The anomalous points in the strata $\Sigma^a_m$ for $m\le n-2$ come from the possibility that, around a singular point $x_\circ$, the function $(u-p_{*,x_\circ})|_{B_r(x_\circ)}$ behaves as $\varepsilon_r q$, where:\\
 - $\varepsilon_r$ is infinitesimal as $r\to 0^+$, but $\varepsilon_r\gg r^\alpha$ for any $\alpha>0$;\\
- $q$ is a nontrivial second order harmonic polynomial.\\
 Although this behavior may look strange, it can  actually happen and our estimate on the size of $\Sigma_m^a$ is optimal.
 Indeed, in the Appendix we construct examples of solutions for which ${\rm dim}(\Sigma_a^m)=m-1$. 
\end{enumerate}
\end{enumerate}
\end{remark}

We now make some general comments on Theorem \ref{thm:main}.
\begin{remark}
\label{rmk:main thm}
\begin{enumerate}
\item Our result on the higher dimensional stratum $\Sigma_{n-1}$ extends the result of \cite{W99} to every dimension, and improves it in terms of the regularity. Actually, as shown in Theorem \ref{thm:C2}, for any $m =1,\ldots,n-1$ we can cover $\Sigma_{m}$ with countably many $C^2$ $m$-dimensional manifolds,
up to a set of dimension at most $m-1$.
\item
The last part of the statement in the case $(n\geq 3)$-(b) was recently proved in \cite{CSV17}. Here we reobtain the same result as a simple consequence of our analysis (see the proof of Theorem \ref{thm:main}).
\item
As we shall see,
the higher regularity of the free boundary stated in the previous theorem comes with a higher regularity of the solution $u$ around singular points. More precisely, $\Sigma$ being of class $C^{k,\alpha}$ at some singular point $x_\circ$ corresponds to $u$ being of class $C^{k+1,\alpha}$ at such point.
\item 
The fact that $\Sigma_m^a$ is relatively open implies that if $x_\circ\in \Sigma_m^a$ then $B_\rho(x_\circ)\cap \Sigma_m^a=B_\rho(x_\circ)\cap \Sigma_m$ for $\rho>0$ small.
In particular
$\dim_\HH \big(B_\rho(x_\circ)\cap \Sigma_m\big)\le m-1
$ ($\leq n-3$ if $m=n-1$).
In other words, the whole stratum $\Sigma_m$ is lower dimensional near anomalous points.
\item[(5)] In \cite{Sak91,Sak93}, Sakai proved very strong structural results for the free boundary in dimension $n=2$. However, his results are very specific to the two dimensional case with analytic right hand side, as they rely on complex analysis techniques. On the other hand, all the results mentioned before \cite{C77,C98,W99,CSV17} are very robust and apply to more general right hand sides. Analogously, also our techniques are robust and can be extended to general right hand sides. In addition, our methods can be applied to the study of the regularity of the free boundary in the parabolic case (the so-called Stefan problem), a problem that cannot be studied with complex variable techniques even in dimension two.
 \end{enumerate}
 \end{remark}

\noindent
{\bf Strategy of the  proof of Theorem \ref{thm:main}.}
The idea of the proof is the following:
let $0$ be a singular free boundary point. As shown in \cite{C98,M03} $u$ is $C^2$ at 0, namely there exists a second order homogeneous polynomial $p_*$, with $D^2p_*\geq 0$ and $\Delta p_*=1$, such that $u(x)=p_*(x)+o(|x|^2)$. In order to obtain our result, our goal is to improve the convergence rate $o(|x|^2)$ into a quantitative bound of the form $O(|x|^{2+\gamma})$ for some $\gamma>0$.
In particular, to obtain $C^{1,1}$ regularity of the singular set we would like to show that $\gamma\geq 1$.

Using motononicity formulae due to Weiss and Monneau, we are able to prove that Almgren frequency function is monotone on $w:=u-p_*$ (this result came as a complete surprise to us, as the Almgren frequency formula has never been used in the classical obstacle problem). This allows us to perform blow-ups around $0$ by considering limits of 
$$
\tilde w_r(x):=\frac{w(rx)}{\|w(r\,\cdot\,)\|_{L^2(\partial B_1)}}\qquad \text{as $r\to 0$},
$$
and prove that if $\lambda_*$ is the value of the frequency at $0$ then $u(x)=p_*(x)+O(|x|^{\lambda_*})$. 
Although it is easy to see that $\lambda_*\geq 2$, it is actually pretty delicate ---and actually sometimes impossible--- to exclude that $\lambda_*=2$ (note that, in such a case, we would get no new informations with respect to what was already known). Hence our goal is to understand the possible value of $\lambda_*.$

To this aim, we consider $q$ a limit of $\tilde w_r$ and, exploiting the monotonicity of the frequency, we prove that $q\not \equiv 0$, $q$ is $\lambda_*$-homogeneous, and $q\Delta q\equiv 0$.

Then we distinguish between the two cases $m=n-1$ and $m\leq n-2$. While in the latter case we can prove that $q$ is harmonic (therefore $\lambda_* \in \{2,3,4,\ldots\}$), in the case $m=n-1$ we prove that $q$ is a solution of the so-called ``Signorini problem'' (see for instance \cite{AC04,ACS08}).
In particular, when $n=2$, this allows us to characterize all the possible values of $\lambda_*$ in dimension $2$ (as all global two-dimensional homogenous solutions are classified). Still, this does not exclude that $\lambda_*=2$.
As shown in Proposition \ref{propblowup} this can be excluded in the case $m=n-1$, while the examples constructed in the Appendix show that $\lambda_*$ may be equal to $2$ if $m\leq n-2$.
To circumvent this difficulty, a key ingredient in our analysis comes from Equation \eqref{howisD2q} which shows that, whenever
$\lambda_*=2$, some strong relation between $p_*$ and $q$ holds. Thus, our goal becomes to prove that this relation cannot hold at ``too many'' singular points.

In order to estimate the size of the set where $\lambda_*<3$, we first consider the low-dimension cases $n=2$ and $n=3$, and then we develop a Federer-type dimension reduction principle to handle the case $n \geq 4$. Note that the Federer dimension reduction principle is not standard in this setting, the reason being that if $x_0$ and $x_1$ are two different singular points, then the blow-ups at such points come from different functions, namely $u-p_{*,x_0}$ and $u-p_{*,x_1}$. Still, we can prove the validity of a dimension reduction-type principle allowing us to conclude that, at most points, $\lambda_*\geq 3$.
This proves the main part of the theorem.

Then, to show that $\Sigma_{n-1}$ is contained in a $C^{1,\alpha_\circ}$-manifolds we prove that $\lambda_*\geq 2+\alpha_\circ>2$ at all points in $\Sigma_{n-1}$. Also, the $C^{1,\log^{\epsilon_\circ}}$ regularity of $\Sigma_m$ for $m\leq n-2$ comes a simple consequence of our analysis combined with Caffarelli's asymptotic convexity estimate \cite{C77}.

Finally, the $C^2$ regularity in two-dimensions requires a further argument based on a new monotonicy formula of Monneau-type.
\\

The paper is organized as follows.

In Section \ref{sect:blow up} we  introduce some classical monotonicity quantities, as well as some variants of them that will play a crucial role in our analysis. In particular,
we prove the validity of a Almgren's monotonicy-type  formula.
Then, given a singular free boundary point $x_\circ$, we investigate the properties of the blow-ups of $u(x_\circ+\cdot)-p_{*,x_\circ}$. 

In Section \ref{sect:proof} we continue our analysis of
the possible homogeneities of the blow-ups and show the validity of a Federer-type reduction principle.
These results, combined with
the ones from Section \ref{sect:blow up}, allow us to prove Theorem \ref{thm:main} in dimensions $n\ge3$, as well as the $C^{1,1}$ regularity of $\Sigma_1$ in dimension $n=2$. The proof of the $C^2$ regularity of $\Sigma_1$ for $n=2$ is postponed to Section \ref{sect:C2}.

In the final Appendix we build solutions exhibiting anomalous points that show the sharpness of Theorem \ref{thm:main}(b).
\\

{\it Acknowledgments:} both authors are supported by ERC Grant ``Regularity and Stability in Partial Differential Equations (RSPDE)''.

\section{Notation, monotonicity formulae, and blow-ups}
\label{sect:blow up}
Let us denote
\begin{equation}
\label{eq:def M}
\mathcal M :=  \big\{ \mbox{symmetric $n\times n$ nonnegative definite matrices $B$  with ${\rm tr}\, B=1$}\big\} 
\end{equation} 
and
\[\mathcal P := \bigg\{ p(x) = { \frac 1 2} x\cdot Bx\  :\   B\in \mathcal M \bigg\}.\]
Given a singular free boundary point $x_\circ$, we denote 
\[ p_{*,x_\circ}(x) = \lim_{r\to 0} r^{-2}u(x_\circ+ rx)\]
(the existence of this limit is guaranteed by \cite{C98}, see also \cite{M03}). 
Note that $\Delta p_{*,x_\circ}\equiv 1$, hence $p_{*,x_\circ}\in \mathcal P$. When $x_\circ=0$, we will sometimes simplify the notation to $p_*$.

Throughout the paper we will assume that $u\not\equiv p_*$ in $B_1$, as otherwise Theorem \ref{thm:main} is trivial. 

\subsection{Weiss, Monneau, and Almgren frequency formula}

In this section we assume that $x_\circ=0$ is a singular point. The goal of the section is to prove that,
for any given $p\in \mathcal P,$
the Almgren frequency formula 
\[ 
 \phi(r,w) :=\frac{r^{2-n}\int_{B_r} |\nabla w|^2}{ r^{1-n}\int_{\partial B_r} w^2},\qquad
 w :=  
u-p, 
\]
is monotone nondecreasing in $r$.
(Note that, since by assumption $u \not\equiv p_*$, then $w := u-p\not\equiv 0$ for any $p \in \mathcal P$ and $\phi(r,w)$ is well defined.)

To this aim, we first recall the definition of the Weiss function
\[
W(r,u) := \frac{1}{r^{n+2}} \int_{B_r} \Bigl(|\nabla u|^2 +2u\Bigr)  -\frac{2}{r^{n+3}}\int_{\partial B_r} u^2. 
\]

\begin{proposition}[Weiss monotone function \cite{W99}] \label{prop:Weiss}
If $0$ is a singular point then
\[
\frac{d}{dr} W(r,u) \ge 0
\]
and
\[
W(0^+, u) = \frac{\mathcal H^{n-1}(\partial B_1)}{2n(n+2)}=W(r,p)\qquad \forall\,p \in \mathcal P,\,\forall \,r>0.
\]
\end{proposition}



To prove the monotonicity of $\phi$ we will use several times the following observation:
\begin{remark}\label{remsign} Since  $\Delta u = \Delta p =1$  in $\{u>0\}$, we have
\[
w\Delta w = 
\begin{cases} 
0  							&\mbox{in } \{u>0\}
\\
p\Delta p = p\ge 0 \quad 	&\mbox{in } \{u=0\}.
\end{cases}
\]
A short way to write this is
\begin{equation}\label{wLapw}
w\Delta w = p\chi_{\{u=0\}} \ge 0.
\end{equation}
\end{remark}

We also need the following auxiliary result, that is essentially due to Monneau \cite{M03}.
\begin{lemma}\label{lemWeiss}
Let $0$ be a singular point, $p \in \mathcal P$, and $w:=u-p$. Then
\begin{equation}
\label{eq:mon1}
\frac{1}{r^{n+2}}\int_{B_r} |\nabla w|^2\geq \frac{2}{r^{n+3}}\int_{\partial B_r} w^2
\end{equation}
and
\begin{equation}
\label{eq:mon2}
 \frac{1}{r^{n+3}} \int_{\partial B_r} w( x\cdot\nabla w-2w) \ge \frac{1}{r^{n+2}}\int_{B_r} w\Delta w\ge 0 
\end{equation}
for all $r>0$.
\end{lemma}
\begin{proof}
Since  $W(0^+,u) =  W(r,p)$ for all $r>0$ (see Proposition \ref{prop:Weiss}) and $\Delta p\equiv 1$, we have
\[
\begin{split}
0 &\le W(r,u) -W(0^+,u) 
=W(r,u) - W(r,p)
\\
&=\frac{1}{r^{n+2}}\int_{B_r} \Bigl(|\nabla w|^2 + 2\nabla w\cdot \nabla p +2w\Bigr) -\frac{2}{r^{n+3}}\int_{\partial B_r} \Bigl(w^2 + 2w p\Bigr)
\\
&= \frac{1}{r^{n+2}}\int_{B_r} |\nabla w|^2  -\frac{2}{r^{n+3}}\int_{\partial B_r} w^2 + \frac{2}{r^{n+3}}\int_{\partial B_r}   w(x\cdot\nabla p -2p)
\\
&= \frac{1}{r^{n+2}}\int_{B_r} |\nabla w|^2 -\frac{2}{r^{n+3}}\int_{\partial B_r} w^2 ,
\end{split}
\]
where we used that $p$ is $2$-homogeneous (hence $x\cdot \nabla p=2p$). This proves \eqref{eq:mon1}.

Now, since
$$
\frac{1}{r^{n+2}}\int_{B_r} |\nabla w|^2=
 \frac{1}{r^{n+2}}\int_{B_r} - w\Delta w + \frac{1}{r^{n+3}}\int_{\partial B_r} w\,x\cdot \nabla w,
$$
\eqref{eq:mon2} follows from \eqref{eq:mon1}
and \eqref{wLapw}.
\end{proof}

We can now state and prove the monotonicity of the Almgren frequency function.
We remark that the fact that $\phi$ is monotone for all $p \in \mathcal P$ (and not only with $p=p_*$) will be crucial in the proof of Theorem \ref{thm:main}.

\begin{proposition}[Almgren frequency formula]\label{lemAlmgren}
Let $0$ be a singular point, $p \in \mathcal P$, and $w:= u-p$. Then
\[
\frac{d}{dr} \log \phi(r,w) \ge\frac 2 r  \frac{
\left(r^{2-n} \int_{B_r} w\Delta w\right)^2 }{ r^{2-n} \int_{B_r}|\nabla w|^2  \ r^{1-n} \int_{\partial B_r} w^2 }\ge 0.
\]
\end{proposition}

\begin{proof}[Proof of Proposition \ref{lemAlmgren}]
Let us introduce the adimensional quatities
\[D(r) := r^{2-n} \int_{B_r} |\nabla w|^2 = r^2\int_{B_1} |\nabla w|^2(r\,\cdot\,),$$
$$
H(r) := r^{1-n}\int_{\partial B_r} w^2 = \int_{\partial B_1} w^2(r\, \cdot\,),\]
so that $\phi=D/H$.
By scaling it is enough to compute the derivative  of $\phi$ at $r=1$ and prove that it is nonnegative. 

Using lower indices to denote partial derivatives (so $w_i=\partial_{x_i}w,$ $w_{ij}=\partial^2_{x_ix_j}w$, etc.),
we have
\begin{equation}\label{zero}
\frac{d}{dr} \log \phi = \frac{D'}{D}- \frac{H'}{H} 
\end{equation}
where 
\begin{equation}\label{first}
\begin{split}
D'(1) &=  \sum_{i,j} \int_{B_1} 2w_i x_j  w_{ij}  + 2D(1)
\\
&= \sum_{i,j}\int_{\partial B_1}2w_i x_j w_{j} \nu_i  -\sum_{i,j} \int_{B_1}  2(w_{i}x_j)_i w_j +2D(1)
\\
&= 2\int_{\partial B_1} w_\nu ^2 - 2\int_{B_1} \Delta w \,(x\cdot \nabla w) - 2\int_{B_1} |\nabla w|^2 +2D(1)
\\
&= 2\int_{\partial B_1} w_\nu ^2 - 2\int_{B_1} \Delta w \,(x\cdot \nabla w)
\\
&= 2\int_{\partial B_1} w_\nu ^2 - 2\int_{B_1\cap\{u=0\}} \,(x\cdot \nabla p)
\\
&= 2\int_{\partial B_1} w_\nu ^2 - 4\int_{B_1\cap\{u=0\}}  p.
\end{split}
\end{equation}
Here we used that $x\cdot \nabla w|_{\partial B_1}=w_\nu|_{\partial B_1}$ is the outer normal derivative, $\Delta w=-\chi_{\{u=0\}}$, and $x\cdot \nabla p=2p$ (since $p$ is $2$-homogeneous).

On the other hand, recalling \eqref{wLapw}, we have
\begin{equation}\label{secon}
\int_{\partial B_1}w w_\nu = \int_{B_1} w\Delta w + \int_{B_1}|\nabla w|^2
=\int_{B_1\cap\{u=0\}}  p + \int_{B_1}|\nabla w|^2,
\end{equation}
therefore
\begin{equation}\label{third}
\begin{split}
H'(1) = 2\int_{\partial B_1}  w w_\nu &=  2\int_{B_1\cap\{u=0\}}  p + 2\int_{B_1}|\nabla w|^2.
\end{split}
\end{equation}
Hence, combining \eqref{first}, \eqref{secon}, \eqref{third}, and \eqref{zero}, and denoting 
\[I:=\int_{B_1}w\Delta w = \int_{B_1\cap\{u=0\}}  p \ge 0,\]
we obtain
\[
\begin{split}
\frac{d}{dr} \log \phi(1,w) &=  2\left(\frac{\int_{\partial B_1} w_\nu^2 -2I}{\int_{B_1} |\nabla w|^2} - \frac{\int_{\partial B_1}ww_\nu}{\int_{\partial B_1} w^2} \right)
\\
&= 2\, \frac{\big(\int_{\partial B_1} w_\nu^2 -2I \big)\int_{\partial B_1} w^2 - \int_{\partial B_1}ww_\nu \big(\int_{\partial B_1}ww_\nu -I\big)  }{\int_{B_1} |\nabla w|^2 \int_{\partial B_1} w^2} 
\\
&= 2\, \frac{\big\{\int_{\partial B_1} w_\nu^2\int_{\partial B_1} w^2  -\big(\int_{\partial B_1}ww_\nu\big)^2 \big\}  + I \int_{\partial B_1}w(w_\nu-2w) }{\int_{B_1} |\nabla w|^2 \int_{\partial B_1} w^2} 
\end{split}
\]
Observe that the first term inside the brackets is nonnegative by the Cauchy-Schwartz inequality. Also, recalling \eqref{eq:mon2}, we have that
\[
 \int_{\partial B_1} w( w_\nu-2w) \ge \int_{B_1} w\Delta w =I.
\]
Since $I\geq 0$, the result follows.
\end{proof}

Note that, because $r\mapsto\phi(r,w)$ is monotone nondecreasing, it must have a limit as $r\downarrow 0.$
The first observation is that this limit is at least $2$.

\begin{lemma}\label{lemblowup0}
Let $0$ be a singular point, $p\in \mathcal P$, and $w:= u-p$. 
Then $\phi(0^+,w) \geq 2$.
\end{lemma}
\begin{proof}
It suffices to observe that \eqref{eq:mon1} is equivalent to $\phi(r,w)\geq 2$ for all $r>0$.
\end{proof}

A first classical consequence of the frequency formula is the following monotonicity formula:
\begin{lemma}\label{Hincreasing}
Let $0$ be a singular point, $p\in \mathcal P$, and $w:= u-p$. 
Given $\lambda >0$ denote
\[
H_\lambda(r,w) := \frac{1}{r^{n-1+2\lambda}} \int_{\partial B_r} w^2
\]
Then the function
$r\mapsto H_{\lambda}(r,w)$ is nondecreasing
 for all $0\le \lambda \le  \phi(0^+,w)$.
\end{lemma}

\begin{proof}
Denoting
\[w_r(x) := (u-p)(rx).\] we have
\[
\frac{H'_{\lambda}} {H_{\lambda}} (r,w) = \frac{ 2r^{-2\lambda} \int_{\partial B_1} w_r(x) \big(x\cdot \nabla w(rx)\big) -2\lambda r^{-2\lambda-1}  \int_{\partial B_1} w_r ^2 }{ r^{-2\lambda} \int_{\partial B_1} w_r ^2}.
\] 
Using that
\[ 
r\int_{\partial B_1} w_r(x) \big(x\cdot \nabla w(rx)\big)=\int_{\partial B_1} w_r (x\cdot \nabla w_r)  = \int_{B_1} |\nabla w_r|^2 + \int_{B_1} w_r\Delta w_r
\]
and  that $w_r\Delta w_r\ge 0$ (recall \eqref{wLapw}), we obtain 
\[
\frac{H'_{\lambda}} {H_{\lambda}}(r,w) \ge \frac{2\int_{B_1} |\nabla w_r|^2}{r\int_{\partial B_1} w_r ^2 }  - \frac{2\lambda}{r}  = \frac{2}{r} \big(\phi(r,w)-\lambda\big).
\] 
Since $\phi(r,w)\geq \phi(0^+,w) \geq \lambda$ (by Proposition \ref{lemAlmgren}), the result follows.
\end{proof}

\begin{corollary}[Monneau monotonicity formula \cite{M03}]\label{monneau}
Let $0$ be a singular point and let $H_\lambda$ be as in Lemma \ref{Hincreasing}.
The function $H_2(r, u-p)$ is monotone nondecreasing  in $r$, for all $p$ in $\mathcal P$.
\end{corollary}
\begin{proof}
It is a direct consequence of Lemmas \ref{Hincreasing} and \ref{lemblowup0}.
\end{proof}

The following result shows the monotonicity for a modified Weiss function. It is remarkable that the quantity below is monotone for all $\lambda>0$, independently of the value of the frequency. 
\begin{lemma}\label{modifieWeiss}
Let $0$ be a singular point, $\lambda\ge 0$, and $w:= u-p$, where $p\in \mathcal P$. Then the function
\[
W_\lambda(r,w) :=r^{-2\lambda}\bigg( r^{2-n} \int_{B_r} |\nabla w|^2 -  \lambda\, r^{1-n}\int_{\partial B_r} w^2 \bigg)
\]
is monotone nondecreasing in $r$.
\end{lemma}
\begin{proof}
For $0\le \lambda\le  \phi(0^+,w)$ we have $W_\lambda =  (\phi -\lambda) H_\lambda$, the product of two positive nondecreasing functions (thanks to Proposition \ref{lemAlmgren} and Lemma \ref{Hincreasing}), hence $W_\lambda$ is nondecreasing.

The result is more interesting for $\lambda >\phi(0^+,w)$ and it requires a different proof. Indeed, using the notation and calculations from the proof of Proposition \ref{lemAlmgren} we have, for $I := \int_{B_1} w\Delta w  = \int_{B_1\cap\{u=0\}} p \ge 0$,
\[
\begin{split}
W'_\lambda(1) &= D'(1) - \lambda H'(1) -2\lambda \big(D(1)-\lambda H(1)\big) 
\\
& = \bigg(2\int_{\partial B_1} w_\nu ^2- 4 I \bigg) - 2\lambda \int_{\partial B_1}ww_\nu -2\lambda\bigl( D(1)- \lambda H(1)\bigr)
\\
&= \bigg(2\int_{\partial B_1} w_\nu ^2- 4 I \bigg) - 2\lambda \int_{\partial B_1}ww_\nu -2\lambda \bigg(\int_{\partial B_1} ww_\nu -  I \bigg)   +2\lambda^2 H(1)
\\
&= 2 \bigg( \int_{\partial B_1} w_\nu ^2 - 2\lambda\int_{\partial B_1} ww_\nu + \lambda^2 \int_{\partial B_1} w^2\bigg)  +  2(\lambda -2)I 
\\
&= 2  \int_{\partial B_1} (w_\nu -\lambda w)^2   + 2(\lambda -2)I.
\end{split}
\]
Since $\lambda > \phi(0^+,w)\geq 2$ (by Lemma \ref{lemblowup0}), the result follows.
\end{proof}

As a consequence of this result we can prove that, given $\lambda>\lambda_*=\phi(0^+,u-p_*)$, the function $H_\lambda$ blows up at 0. This, combined with the monotonicity of $H_{\lambda_*}$ (see Lemma \ref{Hincreasing}), shows that
\begin{equation}
\label{eq:decay w}
r^{2\lambda}\lesssim \ave_{\partial B_r}w^2 \lesssim r^{2\lambda_*}\qquad \text{ for $r \ll 1$}.
\end{equation}
Note that while this estimate is classical for harmonic functions (since the frequency function is related to the derivative of $H_\lambda$), in our case only an inequality is available (see the proof of Lemma \ref{Hincreasing}) and a different argument is needed.

\begin{corollary}
Let $0$ be a singular point, $w:=u-p_*$, $\lambda_*:=\phi(0^+,w)$, and fix $\lambda>\lambda_*$. Let $H_\lambda$ be as in as in Lemma \ref{Hincreasing}. Then 
$$
H_\lambda(r,w)\to +\infty\qquad \text{as $r\downarrow0$.}
$$
\end{corollary}

\begin{proof}
Assume by contradiction that there exists a sequence $r_k\downarrow0$ such that $H_\lambda(r_k,w)\leq C$ for some constant $C$. Then, taking $\mu \in (\lambda_*,\lambda)$, it follows that  $H_\mu(r_k,w)\to 0$. 
Hence, with the notation of Lemma \ref{modifieWeiss}, this gives (since $W_\mu \geq -\mu \,H_\mu$)
$$
\liminf_{k\to \infty}W_\mu(r_k,w)\geq \liminf_{k\to \infty} -\mu \,H_\mu(r_k,w) =0.
$$
By the monotonicity of $W_\mu$, this implies that $W_\mu(r,w)\geq 0$ for all $r>0$, or equivalently
$$
r^{2-n} \int_{B_r} |\nabla w|^2 \geq  \mu\, r^{1-n}\int_{\partial B_r} w^2\qquad \forall\,r>0.
$$
But this means that $\phi(r,w)\geq \mu$ for all $r>0$, a contradiction to the fact that $\mu>\lambda_*$.
\end{proof}

\subsection{Blow-up analysis}
We now start investigating the structure of possible blow-ups.
\begin{proposition}\label{propblowup}
Let $0$ be a singular point, $w:=u-p_*$, and for $r>0$ small define
\[ w_r(x): = w(rx) ,\qquad 
 \tilde w_r :  = \frac{w_r}{\|w_r\|_{L^2(\partial B_1)}}.
 \]
Let $ L := \{p_*=0\}$,
and  $m\in \{0,1,2,\dots n-1\}$ be the dimension of $L$. Also, let $\lambda_*:=\phi(0^+,w)$.
Then:
\begin{enumerate}
\item[(a)]
For $0\le m\le n-2$ we have $\lambda_*\in \{2,3,4,5,\dots\}$. Moreover,  for every sequence $r_k\downarrow 0$ there is a subsequence $r_{k_\ell}$ such that  $\tilde w_{r_{k_\ell}} \rightharpoonup q$ in $W^{1,2}(B_1)$ as $\ell \to \infty$,  where $q\not\equiv 0$ is a $\lambda_*$-homogeneous harmonic polynomial.  

In addition, if $\lambda_*=2$, then in an appropriate coordinate frame it holds
\begin{equation}\label{howisD2q}
D^2p_*= 
\left(
\begin{array}{ccc|c}
 \mu_1& & &\\
& \ddots & & 0_m^{n-m}\\
 &  &\mu_{n-m}&\\
\hline
 & 0_{n-m}^m &&  0_{m}^{m}
\end{array}
\right)
\quad
\mbox{and} 
\quad
D^2 q= 
\left(
\begin{array}{ccc|c}
 t& & &\\
& \ddots & & 0_m^{n-m}\\
 &  &t&\\
\hline
 & 0_{n-m}^m && -N
\end{array}
\right),
\end{equation}
where $\mu_1,\ldots,\mu_{n-m},t>0$, $\sum_{i=1}^{n-m}\mu_i=1$, and $N$ is a symmetric nonnegative definite $m\times m$ matrix with ${\rm tr }(N) =(n-m)t$.
\item[(b)] For $m=n-1$ we have $\lambda_*\ge 2+\alpha_\circ$, where $\alpha_\circ >0$ is a dimensional constant. Moreover, for every sequence $r_k\downarrow 0$ there is a subsequence $r_{k_\ell}$ such that  $\tilde w_{r_{k_\ell}} \rightharpoonup q$ in $W^{1,2}(B_1)$,  where $q\not\equiv 0$ is a $\lambda_*$-homogeneous solution of the Signorini problem (with obstacle $0$ on $L$):
\begin{equation}\label{TOP}
\Delta q\le 0\quad \text{and} \quad q\Delta q=0 \quad \mbox{in }\R^n, \quad \Delta q=0 \quad \mbox{in }\R^n\setminus  L,  \quad\mbox{and}\quad q\ge 0 \quad \mbox{on } L.
\end{equation}
\end{enumerate}
\end{proposition}

To prove Proposition \ref{propblowup}, we  need the following auxiliary lemmas:
\begin{lemma}\label{charq}
Let $\tilde w_r$ be as in Proposition \ref{propblowup}
and assume that, for some sequence $r_{k_\ell}\downarrow 0$, it holds $\tilde w_{r_{k_\ell}} \rightharpoonup q$ in $W^{1,2}(B_1)$.  Then
\begin{equation}\label{orthogonality}
\int_{\partial B_1} q(p_*-p) \ge  0 \qquad \mbox{for all } p \in \mathcal P.
\end{equation}
\end{lemma}
 
 \begin{proof}
By the definition of $p_*$ it holds that
\[ w_r(x) = (u-p_*)(rx) = o(r^2)\quad \mbox{as } r\downarrow 0.\]
Let us denote $h_r:=  \|w_r\| _{L^2{(\partial B_1)}} =o(r^2)$ and $\ep_r: = h_r/r^2 = o(1)$ as $r\downarrow 0$. 
Note that, by the compactness of the trace operator $W^{1,2}(B_1)\to L^2(\partial B_1)$, we have $\tilde w_{r_{k_\ell}}=w_{r_{k_\ell}}/h_{r_{k_\ell}}\to q$ in $L^2(\partial B_1)$.

By Corollary \ref{monneau} and the definition of $p_*$, for any fixed $p\in \mathcal P$ we have
\[
\int_{\partial B_1}  \bigg(\frac{w_{r}}{r^2} + p_*-p\bigg)^2 = \int_{\partial B_1} \biggl(\frac{u(rx)-p(rx)}{r^2}\biggr)^2 \downarrow  \int_{\partial B_1} (p_*-p )^2\qquad \text{as $r \downarrow 0$.}
\]
Hence, since $r^{-2}w_r=\ep_r\tilde w_r$,
\[
\int_{\partial B_1}  \big( \ep_r\tilde w_r + p_* -p\big)^2 \ge  \int_{\partial B_1} (p_*-p )^2\qquad \forall\,r>0,\,\forall\,p \in \mathcal P.
\]
Developing the squares
and taking $r=r_{k_\ell}$ we get
$$
\ep_{r_{k_\ell}}^2  \int_{\partial B_1}  \tilde w_{r_{k_\ell}}^2 + 2  \ep_{r_{k_\ell}}\int_{\partial B_1}  \tilde w_{r_{k_\ell}} (p_*-p) \ge  0.
$$
Dividing by $\ep_{r_{k_\ell}}$ and letting $\ell \to \infty$ we obtain \eqref{orthogonality}.
\end{proof}

\begin{lemma}\label{lemmatrix}
Let $p_* \in \mathcal P$,
and assume that $q\not\equiv 0$ is a $2$-homogeneous harmonic polynomial satisfying \eqref{orthogonality}. Then, in an appropriate system of coordinate, \eqref{howisD2q} holds.
\end{lemma}

\begin{proof}
Take $p \in \mathcal P$ and define  $A:= D^2 p_* $, $B:= D^2 p$, and $C := D^2q$.  Then, since $x\cdot \nabla q=2q$
and $\Delta q=0$, it follows from \eqref{orthogonality} that
\[
\begin{split}
0\le \int_{\partial B_1} q(p_*-p) &= \frac{1}{2} \int_{\partial B_1} q_\nu(p_*-p) = \frac 1 2 \int_{B_1} \nabla q \cdot \nabla (p_*-p) 
\\ &=  \frac 1 2\int_{B_1} Cx\cdot (A-B) x\, dx = c_n {\rm  tr}\big(C(A-B)\big),
\end{split}
\]
for some dimensional constant $c_n>0$. Hence, since $p\in \mathcal P$ was arbitrary, we  
deduce that (recall \eqref{eq:def M})
\begin{equation}\label{ineq}
{\rm tr}(CA) \ge {\rm tr}(CB) \quad \mbox{ for all }B \in \mathcal M.
\end{equation}
To show that this implies \eqref{howisD2q}, 
let $\boldsymbol{v} \in \mathbb S^{n-1}$ be an eigenvector for $C$ corresponding to its largest eigenvalue $\nu_{\max}>0$, and choose $B:=\boldsymbol{v}\otimes \boldsymbol{v}$. Then, since $A\geq 0$ and ${\rm tr} (A)=1$, \eqref{ineq} yields
$$
\nu_{\max}= {\rm tr}(\nu_{\max} {\rm Id} A) \geq {\rm tr}(CA) \geq {\rm tr}(CB)= \nu_{\max}.
$$
Thus
$$
{\rm tr}([\nu_{\max} {\rm Id} - C]A)=0,
$$
and because both $A$ and
$\nu_{\max} {\rm Id} - C$ are symmetric and nonnegative definite, we deduce that the kernels of these two matrices decompose orthogonally $\R^n$. In addition, if we set $L=\{p_*=0\}={\rm ker}(A)$, then $(\nu_{\max} {\rm Id} - C)|_{L^\perp}\equiv 0$. Thanks to this fact and recalling that ${\rm tr}(C)=0$ (since $q$ is harmonic), 
the result follows easily. 
\end{proof}

We can now prove  Proposition \ref{propblowup}.

\begin{proof}[Proof of Proposition \ref{propblowup}] For the sake of clarity, we divide the proof into several steps. 

\smallskip

- {\em Step 1}. We note that $\{\tilde w_r\}$ is precompact. Indeed, by Proposition \ref{lemAlmgren}  we have
\[
\int_{\partial B_1} \tilde w_r^2  =1 \quad \mbox{and}\quad  \int_{B_1} |\nabla \tilde w_r|^2  = \phi(r) \le \phi(1) < \infty. 
\]
This yields uniform bounds  $\| \tilde w_r \|_{W^{1,2}(B_1)} \le C$ for all $r\in (0,1)$.
As a consequence, given a sequence $r_k\downarrow 0$ there is a subsequence $r_{k_\ell}\downarrow 0$ such that 
\[
\tilde w_{r_{k_\ell}} \rightharpoonup q 
\qquad \mbox{in } W^{1,2}(B_1).
\]
In particular, by the compactness of the trace operator $W^{1,2}(B_1)\to L^2(\partial B_1)$, it follows that
\[
\| q\|_{L^2(\partial B_1)} =1.
\]

\smallskip

- {\em Step 2}.  We prove (a). So, we assume $m\leq n-2$ and we consider $q$ a possible limit of a converging sequence $\tilde w_{r_{k_\ell}}$. We want to prove that $q$ is a harmonic homogeneous polynomial.

We first show that $q$ is harmonic. Note that
\begin{equation}
\label{eq:sign Delta w}
\Delta w_r(x)  =\Delta u(rx) -\Delta p_*(rx)   =   -r^2\chi_{\{u=0\}} (rx)  \le 0,
\end{equation}
hence $\Delta \tilde w_r$ is a nonpositive measure.
Note also that the contact set $\{u(r\,\cdot\,)=0 \}$ converge in the Hausdorff sense to $L=\{p_*=0\}$ as $r\to0$ (this follows from the uniform convergence of $r^{-2}u(rx)$ to $p_*$ as $r\to 0$). This implies that $q$ has a distributional Laplacian given by a nonpositive measure supported in $L$. Since $q\in W^{1,2}(B_1)$ (by Step 1) and $L$ has codimension 2 (and thus it is of zero harmonic capacity) it follows that  $q$  must be harmonic.

Let us prove next that $q$ is homogeneous. To this aim we show that 
\begin{equation}\label{Ngoal}
\lambda_* = \phi(R,q) : =\frac{R^{2-n}\int_{\partial B_R} |\nabla q|^2}{R^{1-n}\int_{\partial B_R}  q^2}\qquad \forall\,R\in (0,1].
\end{equation}
Indeed, by lower semicontinuity of the Dirichlet integral we have 
$$\phi(1,q) \leq \liminf_{\ell\to \infty}\phi(1,\tilde w_{r_{k_\ell}})=\liminf_{\ell\to \infty}\phi(1,w_{r_{k_\ell}})=\liminf_{\ell\to \infty}\phi(r_{k_\ell},w)=\lambda_*.
$$ 
Also, since $q$ is harmonic, it follows that $R\mapsto \phi(R,q)$ is nondecreasing (this follows from the classical Almgren frequency formula, or equivalently from the proof of Proposition \ref{lemAlmgren}), thus $\phi(R,q)\leq \lambda_*$ for all $R\in (0,1]$.

To show the converse inequality we apply Lemma \ref{Hincreasing} to $\tilde w_{r_{k_\ell}}$
and let $\ell\to \infty$ to obtain
\begin{equation}\label{decayl2}
\frac{1}{\rho^{2\lambda_*}}\ave_{\partial B_\rho} q^2 \le \ave_{\partial B_1} q^2  =1.
\end{equation}
But since $q$ is harmonic (so, in particular, $q\Delta q\equiv 0$) we have 
\[
\frac{H'_\lambda(R,q) }{H_\lambda(R,q)} = \frac 2R (\phi(R,q) -\lambda)
\]
(this is a classical identity that also follows from the proof of Lemma \ref{Hincreasing}).
Hence, if it was $\phi(R,q)<\lambda_*$ for some $R\in(0,1)$ then, choosing $\lambda:= \phi(R,q)$, we would have that $H_\lambda$ would be nonincreasing on $(0,R)$. In particular we would find  
\[
\frac{1}{\rho^{2\lambda}}\ave_{\partial B_\rho} q^2 \ge \frac{1}{R^{2\lambda}}\ave_{\partial B_R} q^2  >0 \quad \mbox{for $\rho\in (0,R)$},
\]
which contradicts \eqref{decayl2} for $\rho$ small since $\lambda<\lambda_*$.
Hence, we proved \eqref{Ngoal}.

Note that \eqref{Ngoal} says that the  Almgren frequency formula $\phi(R,q)$ is constantly equal to $\lambda_*$ for all $R\in (0,1]$. As a classical consequence,  $q$ is $\lambda_*$-homogeneous. Hence, since $q$ harmonic, it follows that $q$ is a $\lambda_*$-homogeneous harmonic polynomial with $\lambda_*\in \{2,3,4,5,\dots\}$ (recall that $\lambda_*\geq 2$, see Lemma \ref{lemblowup0}).

Finally, to complete the proof of (a), it suffices to combine Lemmas \ref{charq}  and \ref{lemmatrix} to obtain that  \eqref{howisD2q} holds when $\lambda_*=2$.

\smallskip

- {\em Step 3}. We now prove the first part of (b): if $m = n-1$, then $q$ must be a homogenous solution of the Signorini problem. 

Indeed, let $\tilde w_{r_{k_\ell}} \rightarrow q$ in $L^2(B_1)$.
We first show uniform semiconvexity and Lipschitz estimates  that are of independent interest and will be useful later on in the paper. 
Namely, let us prove the estimate
\begin{equation}\label{semic}
\partial^2_{\boldsymbol e\boldsymbol e} \tilde w_r \ge -C \quad \mbox{in }B_R,   \qquad \forall\, \boldsymbol e\in L\cap \mathbb S^{n-1}, \,\forall\,R<1,
\end{equation}
where  $C= C(n,R)$  ---in particular $C$ is independent of $r$.

For this, given a vector $\boldsymbol e \in \mathbb S^{n-1}$ and $h>0$, let
\[\delta^2_{\boldsymbol e,h} f:= \frac{f(\,\cdot\,+h\boldsymbol e)+f(\,\cdot\,-h\boldsymbol e)-2f}{h^2}\]
 denote a second order incremental quotient. For $\boldsymbol e\in L\cap \mathbb S^{n-1}$ we have $\delta^2_{\boldsymbol e,h} p_*\equiv 0$ (since $p_*$ is constant in the directions of $L$). Thus, since $\Delta u =1$ outside of  $\{u=0\}$ and $\Delta u\le 1$ everywhere, 
\[
\Delta \big(\delta^2_{\boldsymbol e,h} w_r \big)
 =  \frac{\Delta u\big(r (\,\cdot\,+h\boldsymbol e)\big)+\Delta u\big(r (\,\cdot\,-h\boldsymbol e)\big) - 2\Delta u\big(r (\,\cdot\,)\big) }{h^2}
\le 0  \hspace{2mm}\mbox{ in } B_1\setminus \{u(r\,\cdot\,)=0\}.
\]
On the other hand, since $u\ge 0$ we have
\[
\delta^2_{\boldsymbol e,h} w_r = \delta^2_{\boldsymbol e,h}  u(r\,\cdot\,) \ge 0  \quad \mbox{ in } \{u(r\,\cdot \,)=0\}.
\]
As a consequence, the negative part of the second order incremental quotient $(\delta^2_{\boldsymbol e,h} \tilde w_r)_-$ is a (nonnegative) subharmonic function, and so is its limit $(\partial_{\boldsymbol e\boldsymbol e}^2\tilde w_r)_-$
(recall that $u\in C^{1,1}$, hence $\delta^2_{\boldsymbol e,h} \tilde w_r\to \partial_{\boldsymbol e\boldsymbol e}^2\tilde w_r$ a.e. as $h \to 0$).

Therefore, given any radius $R' \in (R,1)$, by the weak Harnack inequality (see for instance \cite[Theorem 4.8(2)]{CC95}) there exists $\epsilon=\epsilon(n) \in (0,1)$ such that
$$
\|(\partial_{\boldsymbol e\boldsymbol e}^2\tilde w_r)_- \|_{L^\infty(B_R)} \le C(n,R,R') \biggl(\int_{B_{R'}} (\partial_{\boldsymbol e\boldsymbol e}\tilde w_r)^\epsilon_-  \biggr)^{1/\epsilon} 
\leq C(n,R,R') \biggl(\int_{B_{R'}} |\partial_{\boldsymbol e\boldsymbol e}\tilde w_r|^\epsilon  \biggr)^{1/\epsilon}.
$$
Also, by standard interpolation inequalities, the $L^\epsilon$ norm (here we use $\epsilon<1$) can be controlled by the weak $L^1$ norm, namely
$$
\biggl(\int_{B_{R'}} |\partial_{\boldsymbol e\boldsymbol e}\tilde w_r|^\epsilon \biggr)^{1/\epsilon} \leq C(n,R')\,\sup_{t>0}t\bigl|\bigl\{|\partial_{\boldsymbol e\boldsymbol e}\tilde w_r| >t\bigr\}\cap B_{R'}\bigr|.
$$
Furthermore, by Calderon-Zygmund theory (see for instance \cite[Equation (9.30)]{GT01}), the right hand side above is controlled by
$\|\Delta \tilde w_r\|_{L^1(B_{R''})}+
\|\tilde w_r\|_{L^1(B_{R''})}$, with $R'' \in (R',1)$.
Finally, since $\Delta \tilde w_r \leq 0$, $\|\Delta \tilde w_r\|_{L^1(B_{R''})}$ is controlled by the $L^1$ norm of $\tilde w$ inside $B_1$: indeed, if $\chi$ is a smooth nonnegative cut-off function that is equal to $1$ in $B_{R''}$ and vanished outside $B_1$, then
\begin{equation}
\label{eq:control Delta w}
\|\Delta \tilde w_r\|_{L^1(B_{R''})}\leq -\int_{B_1}\chi\,\Delta \tilde w_r =
-\int_{B_1}\Delta \chi\, \tilde w_r\leq C(n,R'')\int_{B_1}|\tilde w_r|.
\end{equation}
In conclusion, choosing $R'=\frac{2R+1}{3}$ and $R''=\frac{R+2}{3}$ we obtain 
$$
\|(\partial_{\boldsymbol e\boldsymbol e}\tilde w_r)_- \|_{L^\infty(B_R)}\leq C(n,R)
  \|\tilde w_r\|_{L^1(B_{1})} \leq C(n,R)
$$
(recall that $\tilde w_r$ is uniformly bounded in $W^{1,2}(B_1)\subset L^1(B_1)$, see Step 1),
which proves \eqref{semic}.
 
 Note that, as a consequence of \eqref{semic}, the Laplacian of $\tilde w_r$ in the tangential directions is uniformly bounded from below.
Since $\Delta  \tilde w_r \le 0$ everywhere and $L$ is $(n-1)$-dimensional, this implies a uniform semiconcavity estimate in the direction orthogonal  to $L$, namely
\[
\partial^2_{\boldsymbol e'\boldsymbol e'} \tilde w_r \le C  \qquad \mbox{in } B_R ,\quad  \mbox{ for } \boldsymbol e'\in L^\perp \mbox{ with } |\boldsymbol e'|=1, 
\]
where, as before, $R<1$ and $C= C(n,R)$.

Thanks to the previous semiconvexity and semiconcavity estimates, we deduce in particular a uniform Lipschitz bound:
\begin{equation}
\label{eq:wk Lip}
|\nabla \tilde w_r| \le C(n,R)  \qquad \mbox{in } B_R\quad \forall\,R<1. 
\end{equation}
Hence, the convergence  $\tilde w_{r_{k_\ell}}\rightarrow q$ holds also locally uniformly inside $B_1$.

Now, recall that by Proposition \ref{lemAlmgren} we have
\begin{equation}
\label{eq:Sign}
\begin{split}
r\phi'(r,w) &\ge  \phi(r,w)\, \frac{ \left(r^{2-n} \int_{B_r} w\Delta w\right)^2 }{ r^{2-n} \int_{B_r}|\nabla w|^2  \ r^{1-n} \int_{\partial B_r} w^2 }\\
& =  \left(\frac{r^{2-n} \int_{B_r} w\Delta w }{  r^{1-n} \int_{\partial B_r} w^2 } \right)^2 = \left(\int_{B_1} \tilde w_r\Delta\tilde w_r\right)^2.
\end{split}
\end{equation}
Since
$$
\ave_{r_{k_\ell}}^{2r_{k_\ell}} r \phi'(r,w)\,dr \leq 2\int_{r_{k_\ell}}^{2r_{k_\ell}}  \phi'(r,w)\,dr=2\bigl(\phi(2r_{k_\ell},w)-\phi(r_{k_\ell},w)\bigr) \to 0 \qquad \text{as $\ell \to \infty$}
$$
(because $\phi(r,w)\to \lambda_*$ as $r\to 0$),  
using the mean value theorem we may choose $\bar r_{k_\ell}\in [r_{k_\ell}, 2r_{k_\ell}] $ such that $\bar r_{k_\ell} \phi'(\bar r_{k_\ell},w)\rightarrow 0$ as $\ell \to \infty$.
Hence, thanks to \eqref{eq:Sign} and \eqref{wLapw}, we deduce that, for $\rho_\ell := \bar r_{k_\ell}/ r_{k_\ell} \in [1,2]$,
\[\int_{B_1} \tilde w_{r_{k_\ell}}\Delta\tilde w_{r_{k_\ell}}\leq 
\int_{B_{\rho_\ell}} \tilde w_{r_{k_\ell}}\Delta\tilde w_{r_{k_\ell}} \rightarrow 0.
\]
Since $\Delta\tilde w_{r_{k_\ell}} \rightarrow  \Delta q$ weakly$^*$ as measures inside $B_1$, $w_{r_k} \rightarrow q$ strongly in $C^0_{\rm loc}(B_1)$,
and $\tilde w_r\Delta \tilde w_r\geq 0$, we obtain
\[
\int_{B_{R}} q \Delta q =0\qquad \forall\,R<1,
\]
therefore, letting $R\uparrow 1$, $q\Delta q\equiv 0$ inside $B_1$.

Now, since 
\begin{itemize}
\item  $\Delta w_r \le 0$ is supported on $\{u(r\,\cdot\,)=0 \}$, that converges to $L$ as $r\downarrow 0$
\item $w_r = (u-p_*)( r\,\cdot\, ) = u(r\,\cdot\,)\ge 0$  on $L$
\item $\tilde w_{r_{k_\ell}}\to q$ locally uniformly
\end{itemize}
in the limit we obtain that $\Delta q\le 0$, $\Delta q = 0$ outside of $L$, and $q\ge 0$ on $L$.
This proves that $q\in W^{1,2}(B_1)$ is a solution of the thin obstacle problem \eqref{TOP} inside $B_1$.

The same argument as the one used
in Step 2 for case (a) (which only used that $q\Delta q\equiv 0$)  shows that $q$ is $\lambda_*$-homogeneous inside $B_1$. In particular we can extend $q$ by homogeneity to the whole space, and $q$ satisfies \eqref{TOP} in $\R^n$.

\smallskip

- {\em Step 4}. We conclude the proof of (b) by showing that $\lambda_*\ge 2+\alpha_\circ$ for some dimensional constant $\alpha_\circ>0$.

We argue by compactness.
Observe that any blow-up $q$ satisfies
\begin{equation}\label{properties}
x\cdot\nabla q = \lambda_* q, \quad  \int_{\partial B_1}q^2 =1, \quad \Delta q\le 0, \quad q\Delta q=0,\quad q\ge 0 \mbox{ on }L, \quad q(0)=0.
\end{equation}
Also, by Lemma \ref{charq} we have that \eqref{orthogonality}  holds.
Now, if we had a sequence of functions $q^{(k)}$ satisfying \eqref{properties} with $\lambda_*^{(k)} \downarrow 2$,  then we would find some limiting  function $q^{(\infty)}$ satisfying
\eqref{properties} with $\lambda_*^{(\infty)}=2$ and \eqref{orthogonality}. 
Then $q^{(\infty)}$ would be a $2$-homogeneous solution of the thin obstacle problem and hence a quadratic harmonic polynomial (see for instance \cite[Lemma 1.3.4]{GP09}).
Thus, applying Lemma \ref{lemmatrix} with $m=n-1$ we find that, in an appropriate coordinate system,
\[
D^2p_*= 
\renewcommand\arraystretch{1.5}
\left(
\begin{array}{c|c}
1 &0_{n-1}^1
 \\
\hline
0^{n-1}_1 &  0_{n-1}^{n-1}
\end{array}
\right)
\quad
\mbox{and} 
\quad
D^2 q^{(\infty)}= 
\renewcommand\arraystretch{1.5}
\left(
\begin{array}{c|c}
t &0_{n-1}^1
 \\
\hline
0^{n-1}_1 &  -N
\end{array}
\right)
\]
where $N\ge 0$ with ${\rm tr}(N)=t>0$ (since $\|q^{(\infty)}\|_{L^2(\partial B_1)}=1$). However, since $q^{(\infty)}(0)=0$ and $q^{(\infty)}\ge 0$ on $L=\{p_*=0\} = {\rm ker}(D^2p_*) $, we must have $-N\ge 0$, a contradiction.
\end{proof}

We conclude this section with an interesting observation: the gap between the value of the frequency and $2$ controls the decay of the measure of the contact set (recall that $\phi(0^+,w)\geq 2$, see Lemma \ref{lemblowup0}).

\begin{proposition}\label{decayest}
Let $0$ be a singular point, $w:=u-p_*$,
 and $\lambda_*:=\phi(0^+,w).$
 Then
 $$
 \frac{|\{u=0\}\cap B_r|}{|B_r|}\leq Cr^{\lambda_*-2}\qquad \forall\,r>0.
 $$
In addition, the constant $C>0$ can be chosen uniformly at all singular points in a neighborhood of $0$.
\end{proposition}

\begin{proof}
Let
$w_r$ and $\tilde w_r$ be defined as in the statement of Proposition \ref{propblowup}.
Since $\tilde w_r$ is bounded in $W^{1,2}(B_1)$ (see Step 1 in the proof of Proposition \ref{propblowup}) and $\Delta w_{r}\leq 0$ (see \eqref{eq:sign Delta w}), we can bound the mass of $\Delta \tilde w_k$ inside $B_{1/2}$ by considering a smooth nonnegative cut-off function $\chi$ that is equal to $1$ in $B_{1/2}$ and vanished outside $B_1$, and then argue as in \eqref{eq:control Delta w}.
In this way we get
$$
\int_{B_{1/2}}|\Delta \tilde w_r|\leq C\|\tilde w_r\|_{L^1(B_1)}\leq C.
$$
But since 
$$
|\Delta \tilde w_r|=r^2\frac{\chi_{\{u(r\,\cdot\,)=0\}}}{\|w_r\|_{L^2(\partial B_1)}}
$$
and $\|w_r\|_{L^2(\partial B_1)}\leq Cr^{\lambda_*}$ (see \eqref{eq:decay w}),
we conclude that
$$
r^{2-\lambda_*}\frac{|\{u=0\}\cap B_{r/2}|}{|B_{r/2}|}=r^{2-\lambda_*}\frac{|\{u(r\,\cdot\,)=0\}\cap B_{1/2}|}{|B_{1/2}|}\leq C,
$$
as desired.
\end{proof}

Note that the density bound is actually stronger around points corresponding to lower dimensional strata $\{\Sigma_m\}_{1\leq m \leq n-2}$. Indeed, since $\Delta \tilde w_r \leq 0$ and any limit of $\tilde w_r$ is harmonic (see Proposition \ref{propblowup}(a)), it follows that 
$\int_{B_{1/2}}|\Delta \tilde w_r| \to 0$
as $r\to 0$,
so in this case the constant $C$ appearing in the statement can be replaced by $o(1)$.

\begin{remark}
In the case when $0 \in \Sigma_{n-1}$, we can actually prove a stronger estimate, namely that $\{u=0\}\cap B_r$ is contained in a $r^{\lambda_*-1}$-neighborhood of $L=\{p_*=0\}$.
To show this, note that \eqref{eq:wk Lip}
implies that
\begin{equation} \label{Lipchitzest}
|\nabla \tilde w_r|\leq C\qquad \text{in $B_{1/2}$},\quad  \forall\,r>0,
\end{equation}
or equivalently
\begin{equation}
\label{eq:grad u p}
|\nabla u(x)-\nabla p_*(x)|\leq C\frac{\|w_r\|_{L^2(\partial B_1)}}{r}
\qquad \forall\, x\in B_{r/2}.
\end{equation}
Observe now that $\nabla u= 0$ on $\{u=0\}$ (since $u\in C^{1,1}$ and $u\geq 0$). Also, since $p_*(x)=\frac12 (\boldsymbol e\cdot x)^2$ for some $\boldsymbol e \in \mathbb S^{n-1}$,
$$
|\nabla p_*(x)|=\,{\rm dist}(x,L)\qquad \forall\,x \in \R^n.
$$
Hence, it follows by \eqref{eq:grad u p} that
\[
{\rm dist}(x,L)\leq C\frac{\|w_r\|_{L^2(\partial B_1)}}{r}
\qquad \forall\, x\in B_{r/2}\cap \{u=0\}.
\]
Since $\|w_r\|_{L^2(\partial B_1)} \leq Cr^{\lambda_*}$ (see \eqref{eq:decay w}),
we conclude that
\begin{equation}
\label{eq:dist L}
{\rm dist}(x,L)\leq Cr^{\lambda_*-1}
\qquad \forall\, x\in B_{r/2}\cap \{u=0\}.
\end{equation}
\end{remark}

\section{Proof of Theorem \ref{thm:main}}  \label{sect:proof} 

In this section we prove Theorem \ref{thm:main}. This will require a fine analysis of the possible values of the frequency at singular points.
We begin with the simple case $n=2$.
\begin{lemma}\label{possiblefreq}
Let $n=2$ and $0$ be a singular point in $\Sigma_1$. Then $\lambda_* := \phi(0^+, u-p_*)$ belongs to the set 
\[
 \{3, 4,5,6,\ldots\} \cup { \textstyle \big\{\frac 7 2, \frac{11}{2}, \frac {15}{ 2},  \frac {19}{ 2}, \ldots\big\}  }.
 \]
 In particular $\alpha_\circ \geq 1$ (here $\alpha_\circ$ is as in Proposition \ref{propblowup}(b)).
\end{lemma}
\begin{proof}
From Proposition \ref{propblowup}(b) we have that any possible blow-up $q$ is a $\lambda_*$-homogeneous solutions of the Signorini problem in two dimensions with obstacle $0$ on $L$ and $\lambda_*\ge 2+\alpha_\circ>2$. In dimension two, homogeneous solutions to the Signorini problem that are symmetric with respect to $L$ are completely classified via a standard argument by separation of variables, and the set of their possible homogeneities is
\[
 \{1,2,3,4,5\ldots\} \cup { \textstyle \big\{\frac 3 2, \frac 7 2, \frac{11}{2}, \frac {15}{ 2}, \frac {19}{ 2}, \ldots\big\} }
 \]
 (see for instance \cite{FS17}).
 In our case $q$ may also have a odd part. However, the odd part is easily seen to be harmonic, hence its possible homogeneity belongs to the set 
 \[
 \{1,2,3,4,5,\ldots\} .
 \]
 In conclusion
 $$
 \lambda_* \in  \{1,2,3, 4,5,\ldots\} \cup { \textstyle \big\{\frac 3 2,\frac 7 2, \frac{11}{2}, \frac {15}{ 2},  \frac {19}{ 2}, \ldots\big\}  }.
 $$
Since $\lambda_*>2$, the lemma follows.
\end{proof}

As explained in the introduction, our main goal is to prove that the set of points with frequency less than $3$ is small. For this, we need to understand what happens when too many singular points accumulate around another singular point. This is the purpose of the next two lemmata: the first concerns the case $m\leq n-2$, and the second deals with the case $m=n-1$.

\begin{lemma}\label{codge2}
Let $n\ge 3$ and suppose that  $0$ is a singular point. Assume that $m := \dim(L) \le n-2$, and that 
there is a sequence of singular points $x_k\to 0$ and radii $r_k\downarrow 0$ with $|x_k|\le r_k/2$ such that 
\[
\tilde w_{r_k} := \frac{(u-p_*)(r_k \,\cdot\, )}{\|(u-p_*)(r_k\,\cdot\,)\|_{L^2(\partial B_1)}} \rightharpoonup q
\qquad \text{in $W^{1,2}(B_1)$,}
\]
and $y_k :=\frac{x_k}{r_k} \to y_\infty$. 
Then $y_\infty\in L$ and $q(y_\infty) =0.$
\end{lemma}

\begin{proof}
Since $(u-p_*)(r\,\cdot\,) = o(r^2)$  and $u(r_ky_k)=u(x_k)=0$, it follows that $p_*( r_ky_k) = r^2_kp_*(y_k) = o(r_k^2)$  as $r_k\to 0$, therefore $p_*(y_\infty)=0$. This proves that $y_\infty\in L = \{p_*=0\}$. We now prove that $q(y_\infty)=0$.

Note that, since $q$ is homogeneous (see Proposition \ref{propblowup}), if $y_\infty=0$ then the result is trivial. So we can assume that $|y_\infty|>0$.

We now use that $x_k$ is a singular point for $u$. Thanks to Lemma \ref{lemblowup0}
applied at $x_k$ with $p=p_{*}$, we know that the frequency of $u(x_k+\,\cdot\,)-p_{*}$ is at least $2$,
therefore
\[\phi\big(1/2,  u(r_k(y_k +\,\cdot\,)) -p_*(r_k\,\cdot\,) \big) \ge 2.
\]
(Note that here $p_{*}$ is the quadratic polynomial of $u$ at $0$, not at $x_k$!)
Equivalently, recalling the definition of $\tilde w_{r_k}$, we have
\begin{equation}\label{freqge2}
2\le \frac{1}{2}\frac{\int_{B_{1/2}}\big |\nabla \tilde w_{r_k}(y_k +\cdot ) +  h_{r_k} ^{-1} \nabla \big(p_*(r_k y_k+r_k \,\cdot\, )- p_*(r_k \,\cdot\, )\big )\big|^2} 
{\int_{\partial B_{1/2}}\big |\tilde w_{r_k}(y_k +\cdot ) + h_{r_k}^{-1}\big(p_*(r_k y_k+r_k \,\cdot\, )- p_*(r_k \,\cdot\, ) \big)\big|^2},
\end{equation}
where $h_{r_k} := \|(u-p_*)(r_k\,\cdot\,)\|_{L^2(\partial B_1)}$.
Note that, because $p_*$ is a quadratic polynomial that vanishes on $L$, we have 
\begin{equation}
\label{eq:b c}
h_{r_k}^{-1}\big(p_*(r_ky_k+r_k \,\cdot\,  )- p_*(r_k \,\cdot\, )\big) = c_k + b_k \cdot x
\end{equation}
for some constant $c_k \in \R$ and some vector $b_k \in \R^n$ with $b_k \perp L$.

We now observe that, since $|y_k|\leq 1/2$ we have $B_{1/2}(y_k) \subseteq B_1$, therefore
\[
\int_{B_{1/2}} \big|\nabla\tilde w_{r_k}(y_k + \,\cdot\,) \big|^2 + \int_{\partial B_{1/2}} \big|\tilde w_{r_k}(y_k + \,\cdot\,) \big|^2 \le \|\tilde w_{r_k}\|^2_{W^{1,2}(B_1)} \le C.
\]
We claim that
\[
|c_k| \le C \quad \mbox{and} \quad |b_k|\le C, \qquad \mbox{with $C$ independent of $k$}.
\]
Indeed, if this was false, dividing by $(|c_k|+|b_k|)^2$ both the numerator and the denominator in \eqref{freqge2}, we would obtain 
\[
2\le \frac{1}{2}\frac{\int_{B_{1/2}}\big |\nabla (\varepsilon_k(x) + \bar c_k + \bar b_k \cdot x)\big|^2} {\int_{\partial B_{1/2}}\big | \varepsilon_k(x) + \bar c_k + \bar b_k \cdot x)\big|^2},
\]
where $\bar c_k := c_k/ (|c_k|+|b_k|)$, $\bar b_k := b_k/ (|c_k|+|b_k|)$, and $\int_{B_{1/2}}|\nabla \varepsilon_k|^2+\int_{\partial B_{1/2}}\varepsilon_k^2 \rightarrow 0$. Thus, in the limit we would find 
\[
2\le \frac{1}{2}\frac{\int_{B_{1/2}}\big |\nabla (\bar c_\infty + \bar b_\infty \cdot x)\big|^2} {\int_{\partial B_{1/2}}| \bar  c_\infty  + \bar b_\infty  \cdot x|^2} = \frac{ |\bar b_\infty|^2}{4n|\bar c_\infty|^2 + |\bar b_\infty|^2} \le 1,
\]
a contradiction that proves the claim.

Thanks to the claim, up to a subsequence, $c_k \to c_\infty$ and $b_k\to b_\infty$ as $k \to \infty$, with $b_\infty \perp L$.
Note now that, since $x_k$ is a singular point, it follows by Corollary \ref{monneau} that, for all $\rho\in(0,1/2)$,
\[
\frac{1}{\rho ^4} \int_{\partial B_1} | u(x_k +r_k \rho \, \cdot\,)  - p_*(r_k \rho \, \cdot\,)|^2 \le 2^4\int_{\partial B_1} \big| u\big(x_k +{\textstyle \frac {r_k} 2}  \, \cdot\,\big)  - p_*\big({\textstyle \frac {r_k} 2} \, \cdot\,\big)\big|^2  ,
\]
or equivalently, recalling \eqref{eq:b c},
\begin{equation}\label{control}
\frac{1}{\rho ^4} \ave_{\partial B_\rho} | \tilde w_{r_k}(y_k + x)  + c_k +   b_k\cdot x|^2    \le 2^4\ave_{\partial B_{1/2}} | \tilde w_{r_k}(y_k + x)  + c_k +   b_k\cdot x|^2 .
\end{equation}
Hence, in the limit (note that $B_{\rho}(y_\infty)\subset B_1$ for all $\rho \leq 1/2$)
\begin{equation}\label{controllim}
\frac{1}{\rho^4} \ave_{\partial B_\rho} | q(y_\infty  + x)  + c_\infty +   b_\infty\cdot x |^2   \le 2^4\ave_{\partial B_{1/2}} | q(y_\infty + x)  + c_\infty +   b_\infty\cdot x|^2\qquad \forall\,\rho\in (0,1/2).
\end{equation}
Since $q$ is a homogeneous harmonic function, $y_\infty \in L$ with $|y_\infty|>0$, and $b_\infty$ is orthogonal to $L$, it follows by \eqref{controllim} that 
the gradient of $q$ in the directions of $L$ must vanish at $y_\infty$, namely $\nabla_L q(y_\infty) =0$. Hence, by homogeneity we find  $q(y_\infty)=0$, as desired.
\end{proof}

\begin{lemma}\label{cod1}
Let $n\ge 2$ and suppose that  $0$ is a singular point with $m:=\dim(L) =n-1$. Assume that
there is a sequence of singular points $x_k\to 0$  with $x_k\in \Sigma_{n-1}$, and let $r_k \downarrow 0$ with $ |x_k| \le r_k/2$.
Denote
\[
 \lambda_*:=\phi(0^+,u-p_*)\qquad \mbox{and}\qquad  \lambda_{*,x_k}:=\phi\big(0^+,u(x_k+\,\cdot\,)-p_{*,x_k}\big).
\]
Suppose that 
\[
\tilde w_{r_k} := \frac{(u-p_*)(r_k \,\cdot\, )}{\|(u-p_*)(r_k\,\cdot\,)\|_{L^2(\partial B_1)}} \rightharpoonup q
\qquad \text{in $W^{1,2}(B_1)$,}
\]
and $y_k :=\frac{x_k}{r_k} \to y_\infty$. 
Also, let $\R\boldsymbol  e=L^\perp$ with $|\boldsymbol  e|=1$, and denote by $q^{\rm even}$ and $q^{\rm odd}$ the even and odd part of $q$ with respect to $L$, namely
$$
q^{\rm even}(x) = \frac 1 2 \big\{ q(x)  +q\big(x-2(\boldsymbol e\cdot x)\boldsymbol e\big)  \big\} ,\qquad
q^{\rm odd}(x) = \frac 1 2 \big\{ q(x)  -q\big(x-2(\boldsymbol e\cdot x)\boldsymbol e\big)  \big\}.
$$
Finally, let $\alpha_\circ$ be as in Proposition \ref{propblowup}(b).

Then $y_\infty\in L$, and  for $\lambda := \inf_k\{ \lambda_{*,x_k}\} \ge 2+\alpha_\circ$ we have
\begin{equation}
\label{eq:q even y infty}
\rho^{-2 \lambda} \ave_{\partial B_\rho} q^{\rm even}(y_\infty+ x)^2  \le 
2^{2\lambda} \ave_{\partial B_{1/2}}  q^{\rm even}(y_\infty+ x)^2 \quad \forall\,\rho\in(0,1/2).
\end{equation}
In addition, if $\lambda_*<3$ then
\begin{equation}
\label{eq:q y infty}
\rho^{-2\lambda} \ave_{\partial B_\rho} q(y_\infty+ x)^2  \le 2^{2\lambda} \ave_{\partial B_{1/2}}  q(y_\infty+ x)^2
\qquad \forall\,\rho\in(0,1/2).
\end{equation}
\end{lemma}

\begin{proof}
Let
\[
p_{*,x_k}(x) = \lim_{r\downarrow 0}  r^{-2} u(x_k +rx)
\]
and define the second order harmonic polynomial
\[  P_k (x) := \frac{1}{h_{r_k}} \big(  p_{*,x_k}(r_k x) -p_{*,0}(x_k +r_k x)  \big), \quad \mbox{where }h_{r_k} :=  \| (u-p_*)(r_k \,\cdot\,) \|_{L^2(\partial B_1)}. \]
Since $x_k\in \Sigma_{n-1}$, Proposition \ref{propblowup}(b) yields 
\begin{equation}\label{basic}
\phi\big(r_k/2 , u(x_k + \, \cdot\,) -p_{*,x_k}  \big)\geq \phi\big(0^+ , u(x_k + \, \cdot\,) -p_{*,x_k}  \big) = \lambda_{*, x_k}\ge \lambda\ge 2+\alpha_\circ >2,
\end{equation} 
therefore
\begin{equation}\label{freqge2+alpha}
2+\alpha_\circ 
\le
\frac12 \frac{\int_{B_{1/2}}\big |\nabla u(x_k + r_k\, \cdot\,) - \nabla p_{*,x_k}(r_k\, \cdot\, ) \big|^2} 
{\int_{\partial B_{1/2}}\big |u(x_k + r_k\, \cdot\, ) -  p_{*,x_k}(r_k\, \cdot\,) \big|^2}
=
\frac12
\frac{\int_{B_{1/2}}\big |\nabla \tilde w_{r_k}(y_k +\cdot ) -  \nabla P_k\big|^2} 
{\int_{\partial B_{1/2}}\big |\tilde w_{r_k}(y_k +\cdot ) - P_k \big)\big|^2}
\end{equation}
for all $k$. 

We now claim that
\begin{equation}
\label{eq:bdd P}
|P_k| \le C \qquad \forall\, k.
\end{equation}
Indeed, if the coefficients of $P_k$ are not bounded, then dividing by its maximum in the numerator and the denominator of \eqref{freqge2} we obtain
\[
2+\alpha_\circ\le \frac12\frac{\int_{B_{1/2}}\big |\nabla \varepsilon_k - \nabla \bar P_k |^2} {\int_{\partial B_{1/2}}\big | \varepsilon_k - \bar P_k\big|^2},
\]
where $\bar P_k := P_k /|P_k|$ and $\int_{B_1 }|\nabla \varepsilon_k|^2 \rightarrow 0$, thus in the limit we find 
\begin{equation}
\label{eq:freq P}
2+\alpha_\circ \le\frac12 \frac{\int_{B_{1/2}}  |\nabla \bar P_\infty |^2} {\int_{\partial B_{1/2}} | \bar  P_\infty |^2}
\end{equation}
for some quadratic polynomial $\bar P_\infty.$
Note now that, since $0,x_k\in \Sigma_{n-1}$, we have
\[
p_{*,x_k} = \frac 12 (\boldsymbol e_k\cdot x)^2  \quad \mbox{for some } \boldsymbol e_k \in \mathbb S^{n-1}
\]
and 
\[
p_{*,0} = \frac 12 (\boldsymbol e\cdot x)^2  \quad \mbox{for  } \boldsymbol e \in \mathbb S^{n-1}\cap L^\perp.
\]
Also, up to replacing $\boldsymbol e_k$ with $-\boldsymbol e_k$ if needed,
we have that $\boldsymbol e_k\to\boldsymbol e$ (since $p_{*,x_k}\to p_{*,0}$ as $k \to \infty$).
Thus
\[ 
\begin{split}
P_k (x) &= \frac{1}{h_{r_k}} \big(  p_{*,x_k}(r_k x) -p_{*,0}(x_k + r_k x)   \big) 
\\
&=  \frac{r_k^2}{2h_{r_k}} \big( (\boldsymbol e_k \cdot x)^2 - (\boldsymbol e\cdot (y_k+x))^2 \big)
\\
&=  \frac{r_k^2}{2h_{r_k}} \big( (\boldsymbol e_k \cdot x)^2 - (\boldsymbol e\cdot x)^2  - 2a_k(  \boldsymbol e\cdot x) - a_k^2\big)
\end{split}
\]
where $a_k :=  (\boldsymbol e\cdot y_k) \to 0$ (since $y_k \to y_\infty \in L = \boldsymbol e^\perp$).
Thus, since the coefficients of $\bar P_k := P_k /|P_k|$ are uniformly bounded and $a_k^2 \ll 2a_k$ we must have $\bar P_\infty(0) =0$ and 
therefore
$$
\bar P_\infty(x)  = \bar c_1(\boldsymbol e'\cdot x)(\boldsymbol e \cdot x) + \bar c_2(\boldsymbol e \cdot x),
$$
for some constants $\bar c_1,\bar c_2 \in \R$, where
\[
\boldsymbol e' := \lim_{k\to \infty} \frac{\boldsymbol e_k- \boldsymbol e}{|\boldsymbol e_k- \boldsymbol e|} \in \mathbb S^{n-1}\cap L.
\]
Now, since $\boldsymbol e'\perp\boldsymbol e$,
a direct computation using the formula above yields
\[ 
\begin{split}
\frac12\frac{\int_{B_{1/2}}  |\nabla \bar P_\infty |^2} {\int_{\partial B_{1/2}} | \bar  P_\infty |^2}
=\frac12\frac{\bar c_1^2\int_{B_{1/2}} (\boldsymbol e'\cdot x)^2+\bar c_1^2\int_{B_{1/2}} (\boldsymbol e\cdot x)^2+\bar c_2^2|B_{1/2}|} {\bar c_1^2\int_{\partial B_{1/2}} (\boldsymbol e'\cdot x)^2(\boldsymbol e\cdot x)^2 +\bar c_2^2\int_{\partial B_{1/2}}(\boldsymbol e\cdot x)^2}
=\frac{\frac{1}{2(n+2)}\bar c_1^2+\bar c_2^2} {\frac{1}{4(n+2)}\bar c_1^2 +\bar c_2^2}\leq 2,
\end{split}
\]
a contradiction to \eqref{eq:freq P}.
Hence this proves \eqref{eq:bdd P}, and up to a subsequence $P_k\to P_\infty$ as $k\to \infty$, where $P_\infty$ is a second order harmonic polynomial. In addition, by the discussion above, $P_\infty$ has the form
\begin{equation}\label{2}
P_\infty(x)  = c_1(\boldsymbol e'\cdot x)(\boldsymbol e \cdot x) + c_2(\boldsymbol e \cdot x),
\end{equation}
  where $\boldsymbol e'\perp \boldsymbol e$ and $c_1,c_2 \in \R$.
  
Now, by Lemma \ref{Hincreasing} applied to $u(x_k +r_k \,\cdot\,)-p_{*,x_k}$ and using \eqref{basic}, we have
\[
\rho^{-2\lambda} \ave_{\partial B_\rho} | \tilde w_{r_k} (y_k +\,\cdot\,)  -P_k  |^2   \le 2^{2\lambda}\ave_{\partial B_{1/2}} | \tilde w_{r_k} (y_k +\,\cdot\,)  -P_k   |^2,
\]
for all $\rho\in(0,1/2)$, hence, in the limit,
\begin{equation}\label{controllim3}
\rho^{-2\lambda} \ave_{\partial B_\rho} | q(y_\infty  + \,\cdot\,)  - P_\infty )|^2   \le 2^{2\lambda}\ave_{\partial B_{1/2}} | q(y_\infty + \,\cdot\,)  - P_\infty  |^2.
\end{equation}
Since  $P_\infty$
is odd with respect to $L$ (see \eqref{2}), it follows by 
\eqref{controllim3} that
\[\begin{split}
\rho^{-2\lambda}
\ave_{\partial B_\rho}  q^{\rm even}(y_\infty  + \,\cdot\,)^2
&=\rho^{-2\lambda}
\ave_{\partial B_\rho}
\big| \big( q(y_\infty+\,\cdot\,)-P_\infty\big)^{\rm even} \big|^2\\
&\leq \rho^{-2\lambda}
\ave_{\partial B_\rho}
\big| q(y_\infty+\,\cdot\,)-P_\infty \big|^2\\
& \le 2^{2\lambda} \ave_{\partial B_{1/2}}
\big| q(y_\infty+\,\cdot\,)-P_\infty \big|^2,
\end{split}
\]
where we used that $f \mapsto f^{\rm even}$ is an orthogonal projection in $L^2(\partial B_\rho)$. 
This proves \eqref{eq:q even y infty}.

Assume now that in addition $\lambda_*<3$.
We claim that 
\begin{equation}\label{qodd0}
q^{\rm odd}\equiv 0.
\end{equation} 
Indeed, since the homogeneity of $q$ is at least $2+\alpha_\circ$ (by Proposition \ref{propblowup}(b)), we have 
 $\nabla q(0)=   \nabla q^{\rm odd}(0) = 0$. On the other hand,  $q^{\rm odd}$ is a harmonic function in $\R^n$ with sub-cubic growth at infinity (here we use the assumption $\lambda_*<3$) and vanishing on $L$, thus it must be $q^{\rm odd} = c(\boldsymbol e\cdot x )$ for some $c\in \R$ and hence (since $\nabla q^{\rm odd}(0) =0$)
$q^{\rm odd} \equiv 0$.

Thanks to \eqref{qodd0} we get $q=q^{\rm even}$, so \eqref{eq:q y infty} follows from \eqref{eq:q even y infty}.
\end{proof}

For $n\ge 3$ and $m\in\{1,2,\dots, n-1\}$ we define 
\begin{equation}
\label{eq:def anomalous}
\Sigma^a_{m} : = \big\{ x_\circ \in \Sigma_m\ :\  \phi\big(0^+, u(x_\circ+\,\cdot\,)-p_{*,x_\circ} \big)<3 \big\},
\qquad
\Sigma^g_{m} := \Sigma_m \setminus \Sigma^a_{m}.
\end{equation}
We can now give the key lemmas needed to prove Theorem \ref{thm:main}. We begin by showing that points in $\Sigma_1^a$ are isolated inside $\Sigma$.
\begin{lemma}\label{lem11}
Assume $n\ge 3$. Then $\Sigma_1^a$ is a discrete set.
\end{lemma}
\begin{proof}
Assume by contradiction that $0\in \Sigma_1^a$ and $x_k\to 0$ is a sequence of singular points.
By definition, $0\in \Sigma_1^a$ means that $\dim(L)=1$ (where $L:= \{p_*=0\}$) and that $\lambda_* := \phi\big(0^+, u-p_{*} \big)<3$. Hence,  since $n\ge 3$ we have $m=1\le n-2$, thus Proposition  \ref{propblowup}(a) yields $\lambda_*=2$. 

Let $r_k := 2|x_k|$. By Proposition \ref{propblowup} and Lemma \ref{codge2} we have (up to extracting a subsequence) 
\[
\tilde w_{r_k} \rightarrow q \quad \mbox{in }L^2(B_1) \quad \mbox{and} \quad y_k := \frac{x_k}{r_k}  \to y_\infty \in L\cap\partial B_{1/2}.
\]
where and $q$ is a $2$-homogeneous harmonic polynomial satisfying $q(y_\infty)=0$. In addition, since $\lambda_*=2$ we know that  \eqref{howisD2q} holds. Namely, in an appropriate coordinate frame (recall that $m=1$ here) we have
\begin{equation}\label{howisD2q1}
D^2p_*= 
\left(
\begin{array}{ccc|c}
 \mu_1& & &\\
& \ddots & & 0_1^{n-1}\\
 &  &\mu_{n-1}&\\
\hline
 & 0_{n-1}^1 & &  0
\end{array}
\right)
\quad
\mbox{and} 
\quad
D^2 q= 
\left(
\begin{array}{ccc|c}
 t& & &\\
& \ddots & & 0_1^{n-1}\\
 &  &t&\\
\hline
 & 0_{n-1}^1 && -(n-1)t
\end{array}
\right),
\end{equation}
where  $\mu_1,\ldots,\mu_{n-1},t>0$, $\sum_{i=1}^{n-m}\mu_i=1$.

Note that, since $|y_\infty|=1/2$, $q(y_\infty)=0$, and $y_\infty\in L$, by homogeneity of $q$ we must have $q|_L \equiv 0$.
This contradicts the fact that $D^2q|_{L\otimes L}=-(n-1)t<0$ (see \eqref{howisD2q1})
and concludes the proof.
\end{proof}

In order to estimate the measure of $\Sigma_m^a$ for $m\geq 2$ we need to develop a Federer-type dimension reduction argument.
As a first step we need the following standard result in geometric measure theory, that we prove for convenience of the reader.

Before stating it, we recall some classical definitions.
Given $\beta >0$ and
$\delta \in (0,\infty]$,
the Hausdorff premeasures $\mathcal H^{\beta}_\delta(E)$ of a set $E$ are defined as follows:\footnote{In many textbooks, the definition of $\mathcal H^{\beta}_\delta$ includes a normalization constant\ chosen so that the Hausdorff measure of dimension $k$ coincides with the standard $k$-dimensional volume on smooth sets. However  such normalization constant is irrelevant for our purposes, so we neglect it.}
\begin{equation}
\label{eq:def Haus}
\mathcal H^{\beta}_\delta(E):=\inf\biggl\{\sum_i {\rm diam}(E_i)^\beta\,:\,E\subset \bigcup_i E_i,\,{\rm diam}(E_i)<\delta \biggr\}.
\end{equation}
Then, one defines the $\beta$-dimensional Hausdorff measure $\mathcal H^\beta(E):=\lim_{\delta\to 0^+}\mathcal H^{\beta}_\delta(E)$.
We recall that the Hausdorff dimension can be defined in terms of $\mathcal H^{\beta}_\infty$ as follows:
\begin{equation}
\label{eq:def dim}
{\rm dim}_{\mathcal H}(E):=\inf\{\beta >0\,:\,\mathcal H^{\beta}_\infty(E)=0\}
\end{equation}
(this follows from the fact that $\mathcal H^{\beta}_\infty(E)=0$ if and only if $\mathcal H^{\beta}(E)=0$, see for instance \cite[Section 1.2]{Sim83}).

\begin{lemma}\label{abstract}
Let  $E\subset \R^n$ be a set with $\HH^\beta_\infty(E) >0$ for some $\beta\in (0,n]$.
Then:
\begin{enumerate}
\item[(a)] For $\HH^\beta$-almost every point $x_\circ \in E$, there is a sequence $r_k\downarrow 0$  such that 
\begin{equation}\label{denpt}
\lim_{k\to \infty} \frac{\HH^\beta_\infty(E\cap B_r(x_\circ) )}{r_k^\beta} \ge c_{n,\beta}>0,
\end{equation}
where $c_{n,\beta}$ is a constant depending only on $n$ and $\beta$.  Let us call these points ``density points''.
\item[(b)] Assume that $0$ is a  ``density point'', let $r_k\downarrow 0$ be a sequence along which \eqref{denpt} holds, and define define the ``accumulation set'' for $E$ at $0$ as
\[
\mathcal A=\mathcal A_{E} := \big\{z\in \overline{B_{1/2}} \ : \, \exists\,(z_\ell)_{\ell\ge 1},(k_\ell)_{\ell \ge 1} \text{ s.t.  $z_{\ell}\in r_{k_\ell}^{-1}E \cap B_{1/2}$ and  $z_{\ell}\to z$ } \big\}.
\]
Then 
\[\HH_\infty^\beta(\mathcal A) >0.\]
\end{enumerate}
\end{lemma}
\begin{proof}
Part (a) of the lemma is a standard property of the Hausdorff (pre)measures, see for instance \cite[Theorem 1.3.6(2)]{Sim83} for a proof. We now prove (b).

Assume that $0$ is a density point. Then by (a) we have
\begin{equation}\label{123}
H_\infty^\beta(r_k^{-1}E \cap B_{1/2})
=\frac{H_\infty^\beta(E \cap B_{r_k/2})}{r_k^\beta} \ge 2^{-(\beta+1)} c_{n,\beta}>0 \quad \mbox{for all }k \gg 1.
\end{equation}
Note that the accumulation set $\mathcal A$ is a closed.
Assume by contradiction that $\HH^\beta(\mathcal A)  =0$.
Then, by definition of $\HH^\beta_\infty$, given any $\ep>0$ there exists a countable cover of balls $\{\hat B_i \}$ such that 
\[
\mathcal A \,\subset \,\bigcup_{i\ge1} \hat B_i \quad \mbox{and} \quad  \sum_{i\ge 1} {\rm diam}(B_i)^\beta \le \ep.
\]
Since $\mathcal A \subset \overline{B_1}$ is  compact set, we can find a finite subcover.
In particular, there exists $N \in \mathbb N$ such that
\[
\mathcal A   \, \subset \,\bigcup_{i=1}^N \hat B_i \quad \mbox{and} \quad  \sum_{i=1}^N {\rm diam}(\hat B_i)^\beta \le \ep.
\]
But then, since
\[
 r_{k}^{-1} E \cap \overline{ B_{1/2}}  \subset \bigcup_{i=1}^N \hat B_i
\]
for $k$ large enough\footnote{Otherwise there would be a sequence of points $z_{\ell} \in r_{k_\ell}^{-1} E \cap \overline{ B_{1/2}}\setminus \bigcup_{i=1}^N \cup B_i$, and hence their limit $z$ ---up to a subsequence--- would satisfy at the same time $z\in \mathcal A$ and $z\in \overline{B_{1/2}} \setminus \bigcup_{i=1}^N \hat B_i$, a contradiction.}, by definition of $\HH^\beta_\infty$ we obtain 
\[
 \HH^\beta_\infty(r_{k}^{-1}E \cap \overline{ B_{1/2}}  ) \le \ep,
\]
a contradiction with \eqref{123} if $\ep$ is small enough.
\end{proof}

We can now give an appropriate version of  Lemma \ref{lem11} for the case $m = \dim(L)\in\{2,\ldots, n-2\}$.
\begin{lemma}\label{lem12}
Assume $n\ge 4$ and $m\in \{2,\dots, n-2\}$ . 
Then  $\dim_\HH(\Sigma_m^a) \le m-1$.
\end{lemma}
\begin{proof}
Recalling \eqref{eq:def dim},
we assume by contradiction that $\HH_\infty^\beta(\Sigma_m^a)>0$ for some $\beta>m-1$.
By Lemma \ref{abstract}(a), there is a point $x_\circ\in  \Sigma_m^a$ and $r_k\downarrow 0$ such that 
\[
 r_k^{-\beta} \HH_\infty^\beta(\Sigma_m^a \cap B_{r_k}(x_\circ)) \ge c_{n,\beta}  >0. 
\]
Assume without loss of generality that $x_\circ =0$.
 Hence, since $0\in \Sigma_m^a$ and $m\le n-2$, it follows by \eqref{eq:def anomalous} and Proposition \ref{propblowup}(a) that 
\[
\lambda_* := \phi(0^+, u-p_*) =2
\]
and that, up to extracting a subsequence,
\[
\tilde w_{r_k} \rightarrow q \qquad \mbox{in }L^2(B_1),
\]
where  $q$ is a $2$-homogeneous harmonic polynomial. In addition, since $\lambda_*=2$, we know that in an appropriate coordinate frame $D^2p_*$ and $D^2q$ are given by \eqref{howisD2q}.
Also, applying Lemma \ref{abstract}(b),   we deduce that the ``accumulation set'' $\mathcal A=\mathcal A_{\Sigma_m^a}$ 
satisfies $H_\infty^\beta(\mathcal A)>0$.

We claim that $\mathcal A\subset \overline B_1 \cap  L \cap \{q=0\}$.
Indeed, by definition, a point $y$ belongs to $\mathcal A$ if there are sequences of singular points $x_k \to 0$ and of radii $r_k \downarrow 0$  such that $|x_k|\le r_k$ and  $x_k/r_k \to z$. 
Thus $x_k/(2r_k) \to z/2$, and by Lemma \ref{codge2} we obtain $z/2\in L$ and $q(z/2) =0$.
By homogeneity, this implies that $z\in L \cap \{q=0\}$ as claimed.

Finally we note that $L\cap \{q=0\}$ has dimension at most $m-1$. Indeed, if not this would imply that $q \equiv 0$ on $L$, which
would contradict the fact that ${\rm tr}(D^ 2q|_{L\otimes L})=-(n-m)t<0$ (see \eqref{howisD2q}).

Thus $\HH^{m-1} \big(\overline{B_1}\cap L\cap \{q=0\}\big)<+\infty$, which yields (since $\beta>m-1$)
\[
0<\HH^\beta_\infty(\mathcal A) \le \HH^\beta_\infty\big(\overline B_1 \cap  L \cap \{q=0\} \big)=0,
\]
contradiction.
\end{proof}

We now analyze the size of $\Sigma_{n-1}^a$.
We begin with the case $n=3.$

\begin{lemma}\label{lem21}
Let $n=3$. Then $\Sigma^a_{n-1}$ is a discrete set.
\end{lemma}

\begin{proof}
Let us assume that $0\in \Sigma^a_{n-1}$ and that $x_k\to 0$, where $x_k \in \Sigma_{n-1}$.
By Proposition \ref{propblowup} and by definition of  $\Sigma^a_{n-1}$ we have 
\[
\lambda_* := \phi(0^+, u-p_*)  \in [2+\alpha_\circ,3).
\]
Let $r_k :=2|x_k|$ and note that, by Proposition \ref{propblowup}(b), we have (up to subsequence)
\[
\tilde w_{r_k} \rightarrow q \quad \mbox{and} \quad  \frac{x_k}{r_k} \to z \in \partial B_{1/2} 
\]
where $q$ is a $\lambda_*$-homogeneous solution of the Signorini problem (with zero obstacle on $L$).

Also, since $\lambda_*<3$, it follows by Lemma \ref{cod1} that $z\in L$ and
\[
\rho^{-2(2+\alpha_\circ)} \ave_{\partial B_\rho(z)} q^2 \le 2^{2(2+\alpha_\circ)} \ave_{\partial B_{1/2} (z)} q^2,  \qquad \forall \rho\in (0,1/2)
\]
(note that, since $x_k \in \Sigma_{n-1},$
$\inf_k\lambda_{*,x_k}\geq 2+\alpha_0$ by Proposition \ref{propblowup}(b)).
This implies that $q$, $Dq$, and $D^2q$ vanish at $z$, and that $\lambda^z := \phi(0+,q(z+\,\cdot\,))\ge 2+\alpha_\circ$.

Since $q$ is a solution of Signorini that is homogeneous with respect to the point $0$, it is classical fact (this follows from instance from the monotonicity of the frequency function) that a blow-up at $z\in L$, 
\[q^z = \lim_j q(z+r_j\,\cdot\,) / \|q(z+r_j\,\cdot\,) \|_{L^2(\partial B_1)} \]
has translation symmetry in the direction $z$, and it is $\lambda^z$-homogeneous. 
Thus, since $n=3$, $q^z$ depends thus only on two variables (equivalently, it has 2-dimensional symmetry). 
Since homogeneous 2-dimensional solutions of Signorini are completely classified (see the proof of Lemma \ref{possiblefreq}) we deduce that
\[
\lambda^z \in  \{1,2,3,4,5, \dots\}\cup { \textstyle \big\{\frac 3 2, \frac 7 2, \frac{11}{2}, \frac {15}{ 2},   \,\dots\big\}  }
\]
Recalling that $\lambda^z \ge 2+\alpha_\circ$, we get $\lambda^z \ge 3$.
But then we reach a contradiction since, by monotonicity of the frequency and the fact that the limit as $r\to +\infty$ of the frequency is independent of the point, we get 
\[
3\leq \lambda^z = \phi(0^+,q(z+\,\cdot\,)) \le \phi(+\infty,q(z+\,\cdot\,)) = \phi(+\infty,q)  = \lambda_* <3.
\]
\end{proof}

In order to control the size of $\Sigma_{n-1}^a$ for $n\geq 4$, we shall use the following result on the Signorini problem:

\begin{theorem}[{\cite[Theorem 1.3]{FS17}}]\label{FocSpa}
Let $L\subset \R^n$ be a $(n-1)$-dimensional subspace, and let $q$ be solution of the Signorini problem in $\R^n$ with obstacle $0$ on $L$ (see  \eqref{TOP}).
Then, for all $z$ in the contact set $\{q=0\}\subset L$ it holds
\[
\phi(0^+,q(z+\, \cdot\,))  \in \{1,2,3, 4, \dots\} \cup { \textstyle \big\{\frac 3 2, \frac 7 2, \frac{11}{2}, \frac {15}{ 2},   \,\dots\big\}  }
\]
except for at most a set of  Hausdorff dimension $n-3$.
\end{theorem}

\begin{lemma}\label{lem22}
Let $n\ge4$. Then  $\dim_\HH(\Sigma^a_{n-1})\le n-3$.
\end{lemma}

\begin{proof}
Recalling \eqref{eq:def dim}, assume by contradiction that $\HH_\infty^\beta(\Sigma_{n-1}^a)>0$ for some $\beta>n-3$. Then
by Lemma \ref{abstract}(a) there exists a point $x_\circ\in  \Sigma_{n-1}^a$ and $r_k\downarrow 0$ such that 
\[
 r_k^{-\beta}
 \HH_\infty^\beta(\Sigma_{n-1}^a \cap B_{r_k}(x_\circ)) \ge c_{n,\beta}  >0. 
\]
Without loss of generality we assume that $x_\circ =0$.
Then, since $0\in \Sigma_{n-1}^a$, by \eqref{eq:def anomalous} and Proposition \ref{propblowup}(b) we have
\[
\lambda_* := \phi(0^+, u-p_*) \in [2+\alpha_\circ, 3)
\]
and (up to a subsequence)
\[
\tilde w_{r_k} \rightarrow q \quad \mbox{in }L^2(B_1),
\]
where  $q$ is a $\lambda_*$-homogeneous solution of the Signorini problem with obstacle $0$ on $L$. 
Applying Lemma \ref{abstract}(b),   the ``accumulation set'' $\mathcal A=\mathcal A_{\Sigma_{n-1}^a}$ 
satisfies $H_\infty^\beta(\mathcal A)>0$.
Set
\[
\mathcal S =: \big\{z \in \overline B_1\cap  L  \cap\{ q(z)=0\}  \mbox{ such that } \phi(0^+,q(z+\, \cdot\,)) \ge 2+\alpha_\circ\big\}.
\]
Then, by the same argument as in the proof of Lemma \ref{lem12} we deduce that $\mathcal A\subset \overline B_1\cap  L  \cap\{ q=0\}$.
Also, since $\lambda_*<3$,
as in the proof of Lemma \ref{lem21}
it follows by
Lemma \ref{cod1} that
$\phi(0^+,q(z+\,\cdot\,))\geq 2+\alpha_\circ$ for all $z \in \mathcal A$. Hence,
$$
\mathcal A\subset \mathcal S.
$$
We now note that, for all $z\in \mathcal S$, we have 
\[
\phi(0^+,q(z+\, \cdot\,)) \le   \phi(+\infty,q(z+\, \cdot\,)) =   \phi(+\infty,q) =  \lambda_*<3 
\]
(since $0\in \Sigma^a_{n-1}$). Therefore it follows that 
\[\phi\big(0^+,q(z+\, \cdot\,)\big) \in [2+\alpha_\circ,3)\quad \mbox{for all } z\in \mathcal S,\]
and Theorem \ref{FocSpa} yields $\dim_\HH(\mathcal S)= n-3$. In particular $\HH_\infty^\beta(\mathcal S)= 0$ (since $\beta > n-3$) and we obtain
\[0<\HH_\infty^\beta(\mathcal A)\le \HH_\infty^\beta(\mathcal S)=0,\]
a contradiction.
\end{proof}

We will also need the following version of Whitney's extension theorem (see for instance \cite{Fef09} and the references therein):
\begin{lemma}[Whitney's Extension Theorem]\label{WET}
Let $\beta\in(0,1]$, $\ell \in \N$, $K\subset \R^n$ a compact set, and $f: K\rightarrow \R$ a given mapping.
Suppose that for any $x_\circ \in K$ there exists  a polynomial $P_{x_\circ}$ of degree $\ell$ such that:
\begin{itemize}
\item[(i)]  $P_{x_\circ} (x_\circ) = f(x_\circ)$;
\item[(ii)] $| D^kP_{x_\circ} (x) -D^k P_x(x) | \le  C |x-x_\circ|^{\ell+\beta-k}$ for all  $x\in K$ and $k\in\{0,1,\ldots,\ell\}$, where $C>0$ is independent of $x_\circ$.
\end{itemize}
Then there exists $F:\R^n\to \R$ of class $C^{\ell,\beta}$ such that 
\[
F|_{K}\equiv f \qquad \text{and}\qquad F(x) = P_{x_\circ}(x) + O(|x-x_\circ|^{\ell+\beta}) \quad \forall\, x_\circ \in K. 
\]
\end{lemma}

We now prove that the set of points with frequency $\geq \lambda$ is contained in a $C^{\lambda-1}$-manifold. 
Since the classical argument provided in \cite[Theorem 7.9]{PSU12} only shows that the singular set is locally contained in a countable union of manifolds (while here we claim that locally we need only one manifold), we provide the details of the proof.
 
\begin{lemma} \label{lem:laststep}
Let $n\ge 2$, $m\in\{1,2,\dots, n-1\}$, and $\lambda>2$. Let $\ell\in \N$ and $\beta \in(0,1]$ satisfy  $\ell+\beta=\lambda$,
and define
\[
S_{m, \lambda} := \big\{ x_\circ\in \Sigma_m \ :\  \phi\big(0^+, u(x_\circ +\,\cdot\,)-p_{*,x_\circ} \big) \ge \lambda \big\}.
\]
Then $S_{m, \lambda}$   locally contained in a $m$-dimensional manifold of class $C^{\ell-1,\beta}$.
\end{lemma}
\begin{proof}
We prove the result in a neighborhood of the origin.

We begin by recalling that the singular set $\Sigma=\cup_{m=0}^n\Sigma_m$ is closed (this is a classical fact that follows from the relative openness of the set of regular points, see \cite{C77}).
In addition, we note that the monotonicity of the frequency implies that the map
$$
\Sigma\ni x_\circ \mapsto \phi\big(0^+, u(x_\circ +\,\cdot\,)-p_{*,x_\circ} \big)
$$
is upper semicontinuous, being the monotone decreasing limit (as $r\downarrow 0$) of the continuous functions
$$
\Sigma\ni x_\circ \mapsto \phi\big(r, u(x_\circ +\,\cdot\,)-p_{*,x_\circ} \big), \qquad r>0.
$$
Thanks to these facts we deduce that
\[
S_{\lambda} := \big\{ x_\circ\in \Sigma \ :\  \phi\big(0^+, u(x_\circ +\,\cdot\,)-p_{*,x_\circ} \big) \ge \lambda \big\}
\]
is closed. In particular, if we define the compact set $K:=
S_{\lambda}\cap \overline{B_{1/4}}$, we have that $\overline{S_{m,\lambda}\cap B_{1/4}}\subset K$.

Now, given $x_\circ \in K$, we define 
\[
P_{x_\circ} (x) :=  p_{*, x_{\circ}} (x - x_\circ). 
\] 
We want to show that $K$, $f\equiv 0$, and $\{P_{x_\circ}\}_{x_\circ \in K}$ satisfy the assumptions of Lemma \ref{WET} with $\ell$ and $\beta$ as defined above.

Note that, by Lemma \ref{Hincreasing} and the definition of $S_{\lambda}$, for all $x_\circ \in K$ we have
\begin{equation}
\label{eq:bound}
 \|u(x_\circ + \rho\,\cdot\,)-p_{*,x_\circ}(\rho\,\cdot\,)\|_{L^2(B_1)}  \le (2\rho)^\lambda \big\|u\big(x_\circ + {\textstyle \frac 1 2}\,\cdot\, \big) - p_{*,x_\circ}\big({\textstyle \frac 1 2}\,\cdot\,\big)\big\|_{L^2{(\partial B_1)}}
\end{equation}
for all $\rho\in (0,1/2].$

Now, given $x_\circ,x \in K$,
set $\rho := |x-x_\circ|$ (note that $\rho \leq 1/2$), and for simplicity of notation assume that $x_\circ=0$. 
Then it follows from \eqref{eq:bound}
applied both at $0$ and $x$ that
\begin{equation}\label{triangle!}
\begin{split}
\| (P_{0} -P_{x})(\rho\,\cdot\,)\|_{L^2(B_{1})} &\le  \|u( \rho\,\cdot\,)-P_{0}( \rho\,\cdot\,)\|_{L^2(B_{1})} + \|u(\rho\,\cdot\,)-P_x( \rho\,\cdot\,) \|_{L^2(B_{1})}
\\
&= \big\|u(\rho\,\cdot\,)- p_*( \rho\,\cdot\,)\big\|_{L^2(B_{1})} +\big \|u\big( \rho \cdot \big)-p_{*,x} \big(\rho(\cdot-x/\rho)\big)\big\|_{L^2(B_{1})}
\\
&\le \big\|u(\rho\,\cdot\,)- p_*( \rho\,\cdot\,)\big\|_{L^2(B_{1})} +\big \|u\big(\rho \cdot \big)-p_{*,x} \big(\rho(\cdot-x/\rho)\big)\big\|_{L^2(B_{2}(x/\rho))}
\\
&= \big\|u(\rho\,\cdot\,)- p_*( \rho\,\cdot\,)\big\|_{L^2(B_{1})} +\big \|u\big(x+ \rho \cdot \big)-p_{*,x}( \rho\,\cdot\,)\big\|_{L^2(B_{2})}
\\
&\le C \rho^\lambda  .
\end{split}
\end{equation}
In particular, since the norm $\|\cdot\|_{L^2(B_1)}$ is equivalent to the norm $\|\cdot\|_{C^\ell(B_1)}$  on the space of  quadratic polynomials, we obtain the existence of a constant $C>0$ such that
\[
| D^kP_{x_\circ} (x) -D^k P_x(x) | \le  C |x-x_\circ|^{\ell+\beta-k} \quad  \mbox{fo all } x_\circ,x\in K\mbox{ and }k\in\{0,1, \dots, \ell\}
\]
(recall that $\lambda=\ell+\beta$).
Since $P_{x_\circ}(x_\circ) = 0$ for $x_\circ\in K$, applying Lemma \ref{WET} we find a function $F \in C^{\ell,\beta}(\R^n)$ such that  
\[
F(x) = p_{x_\circ}(x-x_\circ) + O(|x-x_\circ|^{\ell+\beta}) \quad \mbox{for all } x_\circ \in K. 
\]
Therefore 
\[
S_{m,\lambda} \cap B_{1/4}\subset K\subset \{\nabla F =0\} = \bigcap_{i=1}^n \{\partial_{x_i} F=0\}.
\]
Now, if $x_\circ \in S_{m,\ell}\cap B_{1/4}$ then $\dim {\rm ker\,}\big(D^2 F(x_\circ)\big)=\dim {\rm ker\,}\big(D^2 p_{*,x_\circ}(0)\big)=m$. This implies that, up to a change of coordinates, the rank of $D^2_{(x_1,\ldots, x_{n-m})} F(x_\circ)$  is maximal, and we conclude by the Implicit Function Theorem that, in a neighborhood of $x_\circ$,
$\bigcap_{i=1}^{n-m} \{\partial_{x_i} F=0\}$ is a $m$-dimensional manifold of class $C^{\ell-1,\beta}$ that contains $S_{m,\lambda}$.
\end{proof}

We are now ready to prove Theorem \ref{thm:main} (except for the $C^2$ regularity in dimension $2$ that will follow from Theorem \ref{thm:C2} in the next section).

\begin{proof}[Proof of Theorem \ref{thm:main}]
We need to prove:
\begin{itemize}
\item[(a)] For $n=2$,  $\Sigma_1$ is locally  contained in a $C^{2}$ curve.
\item[(b)] For $n\ge 3$,  $\Sigma^g_{n-1}$ is locally contained in a $C^{1,1}$ $(n-1)$-dimensional manifold and
$\Sigma^a_{n-1}$ is a relatively  open subset of $\Sigma_{n-1}$ satisfying  ${\rm dim}_{\mathcal H}(\Sigma^a_{n-1})\leq n-3$ (the latter set is discrete  for $n=3$). 
\item[(c)] For $n\ge 3$,  $\Sigma_{n-1}$ can be locally covered by a $C^{1,\alpha_\circ}$  $(n-1)$-dimensional manifold, for some dimensional exponent $\alpha_\circ>0$. 
\item[(d)]  For $n\ge 3$ and $m=1,\ldots,n-2$, $\Sigma_m^g$ can be locally covered by a $C^{1,1}$
$m$-dimensional manifold and $\Sigma^a_m$ is a relatively open subset of $\Sigma_{m}$ satisfying ${\rm dim}_{\mathcal H}(\Sigma^a_m)\leq m-1$ (the latter set is discrete when $m=1$).
\item[(e)] For $n\ge 3$ and $m=1,\ldots,n-2$,  $\Sigma_{m}$ can be locally covered by a $C^{1,\log^{\ep_\circ}}$ $m$-dimensional manifold, for some dimensional exponent $\ep_\circ>0$. 
\end{itemize}

Throughout the proof, we will use the definition of $S_{m,\lambda}$ given in  Lemma \ref{lem:laststep}.
\smallskip

{\it - Proof of (a).} By Lemma \ref{possiblefreq} we have that $\Sigma_{1} = S_{1,3}$. Thus, applying Lemma \ref{lem:laststep}, we obtain that $\Sigma_{1}$ is locally covered by a $C^{1,1}$ curve. To conclude that $\Sigma_1$ can be covered by a $C^2$ curve, we apply Theorem \ref{thm:C2} from the next section.

\smallskip

{\it - Proof of (b).}  By Lemma \ref{lem22}, the Hausdorff dimension of  $\Sigma^a_{n-1}$ is at most $n-3$. Also, by definition we have $\Sigma^g_{n-1} = S_{n-1,3}$,  thus $\Sigma^g$ can be locally covered by a $C^{1,1}$ $(n-1)$-dimensional manifold, thanks to Lemma \ref{lem:laststep}.
The fact that $\Sigma_{n-1}^g$ is relatively closed in $\Sigma_{n-1}$ is a consequence of the fact that $x_\circ \mapsto \phi\big(0^+, u(x_\circ + \,\cdot\,) -p_{*,x_\circ}\big)$ is upper semicontinuous, as shown in the proof of Lemma \ref{lem:laststep}. In the case $n=3$, Lemma \ref{lem11} gives that $\Sigma_{n-1}^a$ is a discrete set.

\smallskip

{\it - Proof of (c).} By Proposition \ref{propblowup}(b) we have that the whole stratum $\Sigma_{n-1}$ is contained in $S_{n-1,2+\alpha_\circ}$, for some dimensional constant $\alpha_\circ>0$. As a consequence, the whole stratum $\Sigma_{n-1}$ can be covered by a $C^{1,\alpha_\circ}$ $(n-1)$-dimensional manifold.

\smallskip

{\it - Proof of (d).}  By Lemma \ref{lem12}, for $1\le m\le n-2$ the Hausdorff dimension of  $\Sigma^a_{m}$ is at most $m-1$ (in the case $m=1$, Lemma \ref{lem12} gives that $\Sigma_{1}^a$ is a discrete set). Also, since by definition $\Sigma^g_{m} = S_{m,3}$, applying again  Lemma \ref{lem:laststep} we obtain that $\Sigma^g_m$ can be locally covered by a $C^{1,1}$ $m$-dimensional manifold. 
Finally, as in the proof of (d), the relative closedness of $\Sigma^g$ follows from the upper semicontinuity of the frequency.

\smallskip

{\it - Proof of (e).} Let $m\leq n-2$.
We claim that the following estimate holds:
\begin{equation}\label{colomboandco}
\big\| u(x_\circ +r\,\cdot\,)-p_{*,x_\circ})(r \,\cdot\,) \big\|_{L^2(\partial B_1)} \le C r^{2} \log^{-\ep_\circ} (1/r) \quad \forall x_o\in \Sigma_m\cap B_{1/2}, \ \forall \,r\in (0,1/2).
\end{equation}
Observe that it is enough to prove \eqref{colomboandco} at points $x_\circ$ such that 
\[\lambda_{*,x_\circ} :=\phi\big(0^+, u(x_\circ + \,\cdot\,) -p_{*,x_\circ}\big) =2.\]
 Indeed, if $\lambda_{*,x_\circ}>2$ then Proposition \ref{propblowup}(a) yields $\lambda_{*,x_\circ}\ge 3$,
 hence \eqref{colomboandco} trivially holds (actually, with a much stronger estimate) thanks to Lemma \ref{Hincreasing}.
So, without loss of generality, we can assume that $\lambda_{*,x_\circ} =2$. 

Let $M>1$ be a large constant to be fixed later.
By Caffarelli's asymptotic convexity estimate \cite{C77} (see also \cite[Corollary 5]{C98}), we have
\begin{equation}
\label{eq:Caff conv}
D^2 u \ge -C \log^{-\ep_\circ}(1/r) \,{\rm Id} \quad \mbox{in } B_r(x_\circ)\qquad \forall\,x_\circ \in \Sigma,
\end{equation}
for some dimensional exponent $\ep_\circ>0$.
Now, let $a_{r} :=\| r^{-2}u(x_\circ + r\cdot)-p_{*x_\circ}(r\,\cdot\,)\|_{L^2}=o(1)$ and $L_{x_\circ}:=\{p_{*,x_\circ}=0\}$.
Thanks to \eqref{eq:Caff conv} we have
\begin{equation}
\label{eq:Caff conv2}
\partial_{\boldsymbol e\boldsymbol e} \big( r^{-2}u(x_\circ + r\cdot)-p_{*x_\circ}(r\,\cdot\,) \big) = \partial_{\boldsymbol e\boldsymbol e} \bigl(r^{-2}u(x_\circ + r\cdot)\bigr) \ge -C \log^{-\ep_\circ}(1/r) \qquad \text{in $B_1$}
\end{equation}
for all $\boldsymbol e\in L_{x_\circ}\cap \mathbb S^{n-1}$.

Assume by contradiction that
\[ 
a_{r_k} \ge M \log^{-\ep_\circ}(1/r_k)  \quad \mbox{ for some } r_k \downarrow 0.
\]
Then, recalling \eqref{eq:Caff conv2}, for any $\boldsymbol e\in L_{x_\circ}\cap \mathbb S^{n-1}$ we find
\[
\partial_{\boldsymbol e\boldsymbol e } \tilde w_{r_k} = \frac{1}{a_k} \partial_{\boldsymbol e\boldsymbol e} \big( r^{-2}u(x_\circ + r\cdot)-p_{*x_\circ}(r\,\cdot\,) \big)  \ge -\frac{C}{M} \quad \mbox{in } B_1.
\]
Thus, since $\tilde w_{r_{k_\ell}} \rightarrow q$ in $L^2(B_1)$ for some subsequence $r_{k_\ell}$ (see  Proposition \ref{propblowup}(a)), we have
\begin{equation}\label{semiconvlim}
\partial_{\boldsymbol e\boldsymbol e } q   \ge -\frac{C}{M}\quad \mbox{in } B_1, \quad \forall\, \boldsymbol e\in L_{x_\circ}\cap \mathbb S^{n-1}.
\end{equation}
In addition, since $\lambda_{*,x_\circ} = 2$,
Proposition \ref{propblowup}(a) implies that $q$ is a quadratic polynomial satisfying
\[
D^2 q |_L \le  0,  \quad D^2 q |_{L^\perp} \ge 0,  \quad  {\rm tr} (D^2 q)=0,   \quad \mbox{and} \quad \|q\|_{L^2(\partial B_1)} =1.
\]
Thanks to this fact, a simple compactness argument shows that there exists ${\boldsymbol e'}\in L_{x_\circ}\cap \mathbb S^{n-1}$ such that 
\[
\partial_{{\boldsymbol e'}{\boldsymbol e'}} q\leq -c_1<0 \quad \mbox{in } B_1,
\]
for some dimensional constant $c_1>0$. This contradicts \eqref{semiconvlim} for $M$ sufficiently large, thus establishing \eqref{colomboandco}.

Thanks to \eqref{colomboandco},
if we define 
\[
P_{x_\circ}(x) := p_{*,x_\circ}(x-x_\circ)\qquad \forall\, x_0\in \Sigma_m,
\]
the argument in the proof of Lemma \ref{lem:laststep} yields
\begin{equation}\label{controlPs}
| D^kP_{x_\circ} (x) -D^k P_x(x) | \le  C |x-x_\circ|^{2-k} \log^{-\ep_\circ}\big(|x-x_\circ|\big) \quad  \forall x_\circ, x \in \Sigma_m\cap \overline{B_{1/2}} , \ k\in\{0,1, 2\}.
\end{equation}
Hence, by Whitney's Extension Theorem (see \cite{Fef09} and the reference therein) and the argument in the proof of Lemma \ref{lem:laststep}, we conclude that \eqref{controlPs} that $\Sigma_m$ is locally contained in a $C^{1,\log^{\ep_\circ}}$ $m$-dimensional manifold. 
\end{proof}

\section{On third order blow-ups}
\label{sect:C2}

In this section we investigate the uniqueness/continuity of third order blow-ups for points in $\Sigma_{m}$,
and prove that $\Sigma_{m}$
can be covered by $C^2$ manifolds, up to a lower dimensional set (see Theorem \ref{thm:C2} below).

We begin by showing the validity of a third-order almost-monotonicity formula
of Monneau-type for all singular points. 

\begin{lemma}
\label{lem:mon 3}
Let $0$ be a singular point, assume that $\lambda_*:=\phi(0^+,u-p_*)\geq 3,$ and let $q$ be a $3$-homogeneous harmonic polynomial that vanishes on $L:=\{p_*=0\}$. Set $v:=u-p_*-q$,
and let $H_\lambda$ be as in Lemma \ref{Hincreasing}.
Then
$$
\frac{d}{dr}H_3(r,v) \geq -C\bigg\|\frac{q^2}{p_*}\bigg\|_{L^\infty(B_1)},
$$
where $C>0$ is a constant that can be chosen uniformly at all singular points in a neighborhood of $0$.
\end{lemma}
\begin{proof}
Set $w:=u-p_*$, $w_r(x):=r^{-3}w(rx)$, and $v_r(x):=r^{-3}v(rx)=w_r(x)-q(x)$.
Then $H_3(r,v)=H_3(1,v_r)$, and we have
\begin{equation}
\label{eq:der mon 3}
\frac{d}{dr}H_3(r,v)=\frac{d}{dr}H_3(1,v_r)=\frac{2}{r}\int_{\partial B_1}v_r((v_r)_\nu-3v_r).
\end{equation}
We now observe that, because $\lambda_*\geq 3$, it holds
$$
W_3(1,w_r)=(\phi(1,w_r)-3)H_3 \geq 0
$$
(here $W_\lambda$ is as in Lemma \ref{modifieWeiss}).
Also, because $q$ is a $3$-homogeneous harmonic polynomial, one easily checks that
$
W_3(1,q)=0.
$
Hence, similarly to the proof of Lemma \ref{lemWeiss}, we get
\[
\begin{split}
0 &\le W_3(1,w_r) -W_3(1,q) 
\\
&=\int_{B_1} \Bigl(|\nabla v_r|^2 + 2\nabla v_r\cdot \nabla q\Bigr) -3\int_{\partial B_1} \Bigl(v_r^2 + 2v_r q\Bigr)
\\
&= \int_{B_1} |\nabla v_r|^2  -3\int_{\partial B_1} v_r^2 + \int_{\partial B_1}   v_r(x\cdot\nabla q -3q)
\\
&= \int_{B_1} |\nabla v_r|^2 -3\int_{\partial B_1} v_r^2 \\
&= \int_{B_1} -v_r\Delta v_r +\int_{\partial B_1} v_r((v_r)_\nu-3v_r),
\end{split}
\]
where we used that $\Delta q\equiv 0$ and $x\cdot \nabla q=3q$. Thus, recalling \eqref{eq:der mon 3} we obtain
$$
\frac{d}{dr}H_3(r,v)=\frac2{r}\int_{B_1}v_r\Delta v_r=\frac{2}{r^{n+5}}\int_{B_r}v\Delta v.
$$
Now, since $\Delta v=\Delta u-\Delta p_*-\Delta q=0$ inside $\{u>0\}$, we have
$$
v\Delta v=(p_*+q)\chi_{\{u=0\}},
$$
therefore
$$
\frac{d}{dr}H_3(r,v)=\frac2{r^{n+5}}\int_{B_1}v_r\Delta v_r=\frac{2}{r^{n+5}}\int_{B_r\cap\{u=0\}}(p_*+q).
$$
Noticing that
$$
p_*+q \ge p_* -|q| \ge \bigg(\sqrt{p_*}-\sqrt{\frac{|q|}{2p_*}}\bigg)^2-\frac{q^2}{2p_*} \geq -\frac{q^2}{2p_*}
$$
and that $\frac{q^2}{2p_*}$ is a $4$-homogeneous polynomial (this follows from the fact that $q=0$ on $\{p_*=0\}$, hence $q$ is divisible by $\sqrt{p_*}$),
we conclude that
$$
\frac{d}{dr}H_3(r,v)=-\frac{1}{r}\int_{B_1\cap\{u(r\,\cdot\,)=0\}}\frac{q^2}{p_*}\geq -\bigg\|\frac{q^2}{p_*}\bigg\|_{L^\infty(B_1)}\frac{|\{u(r\,\cdot\,)=0\}\cap B_1|}{r}.
$$
Since $\lambda_*\geq 3$,
the result follows by Proposition \ref{decayest}.
\end{proof}

In order to apply the previous result, we need to check the size of the points where any third-order blow-up is harmonic and vanishes on $\{p_*=0\}$.
We begin with the case $n=2$.

\begin{lemma}\label{lemdim20}
Let $n=2$, $0\in \Sigma_1$, and $w:=u-p_*$. Assume that there exists a sequence $x_k\in \Sigma_1$ with $x_k \to 0$, and that
\[
w_{r_k} := \frac{(u-p_*)(r_k \,\cdot\, )}{r_k^3} \rightharpoonup \hat q
\qquad \text{in $W^{1,2}(B_1)$.}
\]
Then $\hat q$ is a $3$-homogeneous harmonic polynomial vanishing on $L:=\{p_*=0\}$ and satisfying $\|\hat q\|_{L^2(\partial B_1)}=H_3(0^+,w)$.
%
\end{lemma}

\begin{proof}
Note that if $\phi(0^+,u-p_*)>3$ then $\|w_r\|_{L^2(B_1)}=o(r^3)$ (see \eqref{eq:decay w}), hence $H_3(0^+,w)=0$ and the result holds with $\hat q\equiv 0$. So we can assume that $\phi(0^+,u-p_*)=3$.

Set
$$
\tilde w_{r_k}:=\frac{w_{r_k}}{\|w_{r_k}\|_{L^2(\partial B_1)}}=\frac{w_{r_k}}{H_3(r_k,w)},
$$
and denote by $q$ a limit point for $\tilde w_{r_k}$. 
Note that $\|\hat q\|_{L^2(\partial B_1)}=1$. Also, since $r\mapsto H_3(r,w)$ is monotone nondecreasing (see Lemma \ref{Hincreasing}) and $\hat q\not \equiv 0$, we deduce that $$\hat q=H_3(0^+,w)\,q.$$
This proves that $\|\hat q\|_{L^2(\partial B_1)}=H_3(0^+,w)$.
To conclude the proof it suffices to prove that $q$ is a $3$-homogeneous harmonic polynomial vanishing on $L:=\{p_*=0\}$.

We know by Proposition \ref{propblowup}(b) that $q$ is a $3$-homogeneous solutions of Signorini, see
\eqref{TOP}.  Also,
applying Lemma \ref{cod1} with $r_k=2|x_k|$, we deduce that $y_k:=\frac{x_k}{r_k}\to y_\infty \in L\cap \partial B_{1/2}$ and that (thanks to \eqref{eq:q even y infty})
$$
\ave_{\partial B_\rho}q^{\rm even}(y_\infty+x)^2 \leq C\rho^6
$$
(note that for $n=2$ we have that $\lambda_{*,x_k}\geq 3$ for all $k$, see Lemma \ref{possiblefreq}).
This implies in particular that $q^{\rm even}$ is $3$-homogeneous both with respect to $0$ and $y_\infty$, hence it must be one dimensional. Since $\Delta q=0$ outside $L$, this implies that
$q^{\rm even}$ is affine on each side of $L$, hence $q^{\rm even}\equiv 0$ (being $q^{\rm even}$ 3-homogeneous).

This proves that $q$ is odd with respect to $L$, so $q$ cannot have a singular Laplacian on $L$. Recalling that $\Delta q=0$ in $\R^2\setminus L$, this proves that $q$ is a $3$-homogeneous harmonic polynomial.
Finally, since $q \geq 0$ on $L$
and $q$ is $3$-homogeneous,
it must be $q|_L\equiv 0$.
\end{proof}

Thanks to the previous result and the Federer-type reduction argument developed in the previous section, we obtain the following:
\begin{lemma}
\label{lem:sigma h}
Let $n\geq 2$, $1\leq m \leq n-1$, and let 
$\Sigma_{m}^{3rd}$
denote the set of singular points $x_\circ \in \Sigma_{m}$ such that $\phi(0^+,u(x_\circ+\,\cdot\,)-p_{*,x_\circ})\geq 3$ and the following holds: for any sequence $r_k\to 0$ such that 
$$
w_{r_k} := \frac{u(x_\circ+r_k\,\cdot\,)-p_{*,x_\circ}(r_k \,\cdot\,)}{r_k^3} \rightharpoonup \hat q
\qquad \text{in $W^{1,2}(B_1)$,}
$$
$\hat q$ is a $3$-homogeneous harmonic polynomial vanishing on $\{p_{*,x_\circ}=0\}$ and satisfying $\|\hat q\|_{L^2(\partial B_1)}=H_3(0^+,u(x_\circ+\,\cdot\,)-p_{*,x_\circ})$.

Then:
\begin{enumerate}
\item[(i)] $\Sigma_{1}\setminus \Sigma_{1}^{3rd}$ consists of isolated points for $n\geq 2$;\\
\item[(ii)]${\rm dim}_\HH(\Sigma_{m}\setminus \Sigma_{m}^{3rd})\leq m-1$ for  $2 \leq m \leq n-1$.
\end{enumerate}
\end{lemma} 

\begin{proof}
Since the argument is similar to the 
ones used in the previous section, we just explain the main steps, leaving the details to the interested reader.

Point (i) for $n=2$ follows immediately from Lemma \ref{lemdim20} and Lemma \ref{lem:mon 3}. The case $n\geq 3$
for $m=1$ follows instead by 
Proposition \ref{propblowup}(a)
and Lemma \ref{codge2}.

Concerning the case $m= n-1\ge 2$,
recalling \eqref{eq:def anomalous} and Lemma \ref{abstract},
one can argue similarly to the proof of Lemma \ref{lem22} to prove that Lemma \ref{lemdim20} applies to all points in $\Sigma_{n-1}^g$ 
that are density points for $\Sigma_{n-1}^g$ with respect to the measure $H^\beta_\infty$, with $\beta>n-2$.  Thus, thanks to Lemma \ref{abstract}(a), we deduce that Lemma \ref{lemdim20}
applies to all points in $\Sigma_{n-1}^g$ up to at most a set of Hausdorff dimension $n-2$.
Since ${\rm dim}_\HH(\Sigma_{n-1}\setminus \Sigma_{n-1}^g)={\rm dim}_\HH(\Sigma_{n-1}^a) \leq n-3$ (see Theorem \ref{thm:main}), this proves (ii) when $m=n-1$.

Analogously, in the case $2 \leq m \leq n-2$,  using Proposition \ref{propblowup}(a) and arguing as in Lemma \ref{lem12},
we deduce that Lemma \ref{lemdim20}
applies to all points in $\Sigma_{m}^g$ up to at most a set of Hausdorff dimension $m-1$. Since ${\rm dim}_\HH(\Sigma_{m}\setminus \Sigma_{m}^g)={\rm dim}_\HH(\Sigma_{m}^a) \leq m-1$ (see Theorem \ref{thm:main}), this concludes the proof of (ii).
\end{proof}

We can now prove the uniqueness and continuity of third-order blow-ups at all points in $\Sigma_m^{3rd}$:
\begin{proposition}
\label{prop:C2}
Let $n\geq 2$, $1\leq m \leq n-1$,
and let $x_\circ \in \Sigma_m^{3rd}$. Then the following limit exists: 
\begin{equation}
\label{eq:3rd limit}
\frac{u(x_\circ+rx)-p_{*,x_\circ}(rx)}{r^3}\rightharpoonup q_{*,x_\circ}(x)\qquad \text{in $W^{1,2}(B_1)$ as $r\to 0$,} 
\end{equation}
where $q_{*,x_\circ}(x)$ is a 3-homogeneous harmonic polynomial vanishing on $\{p_{*,x_\circ}=0\}$ and satisfying $\|q_{*,x_\circ}\|_{L^2(\partial B_1)}=H_3(0^+,u(x_\circ+\,\cdot\,)-p_{*,x_\circ})$.
In addition the above convergence is uniform on compact sets, and the
$$
\text{the map }\Sigma_m^{3rd}\ni x_\circ \mapsto q_{*,x_\circ}\text{ is continuous.}
$$
\end{proposition}
\begin{proof}
Assume $0 \in \Sigma_m^{3rd}$. We first prove the existence of a limit.

Let $q_1$ and $q_2$ be two different limits obtained along two sequences $r_{k,1}$ and $r_{k,2}$ both converging to zero. Up to taking a subsequence of $r_{k,2}$ and relabeling the indices, we can assume that $r_{k,2}\leq r_{k,1}$ for all $k$.
Thus, thanks to Lemma \ref{lem:mon 3}, we have
$$H_3(r_{k,1},w-q_1)\geq H_3(r_{k,2},w-q_1)-C|r_{k,2}-r_{k,1}|\qquad \forall\,k,
$$
for some constant $C$ depending on $q_1$.
Thus, letting $k\to \infty$ we obtain
$$
0=\lim_{k\to \infty} \int_{B_1}(w_{r_{k,1}}-q_1)^2\geq \lim_{k\to \infty} \biggl(\int_{B_1}(w_{r_{k,2}}-q_1)^2-C|r_{k,2}-r_{k,1}|\biggr) =\int_{B_1}(q_2-q_1)^2.
$$
This proves the uniqueness of the limit.

We now prove the continuity of the map $x_\circ \mapsto q_{*,x_\circ}$ at $0 \in \Sigma_m^{3rd}$.
Fix $\ep>0$, and consider a sequence $x_k\in \Sigma_m^{3rd}$ with $x_k\to 0$.
Thanks to \eqref{eq:3rd limit}, there exists a small radius $r_\ep>0$ such that
\begin{equation}
\label{eq:eps}
\int_{\partial B_1}\biggl|\frac{u(r_\ep x)-p_{*,0}(r_\ep x)}{r_\ep^3}- q_{*,0}(x)\biggr|^2 \leq \ep.
\end{equation}
Now, let $R_k:\R^n\to \R^n$ be a rotation that maps the $m$-dimensional plane $L_k:=\{p_{*,x_k}=0\}$ onto $L_0:=\{p_{*,0}=0\}$,
and note that $R_k\to {\rm Id}$ as $k \to \infty$ (this follows by the continuity of $\Sigma_m\ni x\mapsto p_{*,x}$).
Then, since $q_{*,0}\circ R_k$ vanishes on $L_k$, we can apply Lemma \ref{lem:mon 3} at $x_k$ with $q=q_{*,0}\circ R_k$ to deduce that
\begin{align*}
\int_{\partial B_1}|q_{x_k,*} - q_{*,0}\circ R_k|^2 &=\lim_{r\to 0}
 \int_{\partial B_1}\biggl|\frac{u(x_k+r x)-p_{*,x_k}(r x)}{r^3} - q_{*,0}\circ R_k(x)\bigg|^2\\
 & \leq 
  \int_{\partial B_1}\biggl|\frac{u(x_k+r_\ep x)-p_{*,x_k}(r_\ep x)}{r_\ep^3} - q_{*,0}\circ R_k(x)\bigg|^2 +Cr_\ep.
\end{align*}
Note that the constant $C$ above is independent of $k$ since, by the continuity of $p_{*,x_k}$,
$p_{*,x_k}\circ R_k^{-1} \geq p_{*,0}/2$ for $k$ large enough, therefore
$$\bigg\|\frac{(q_{*,0}\circ R_k)^2}{p_{*,x_k}}\bigg\|_{L^\infty(B_1)}\leq 2\bigg\|\frac{q_{*,0}^2}{p_{*,0}}\bigg\|_{L^\infty(B_1)}\qquad \forall\,k\gg 1.$$
Hence, since $R_k\to {\rm Id}$, letting $k \to \infty$ and recalling \eqref{eq:eps} we obtain
\begin{align*}
\limsup_{k\to \infty}\int_{\partial B_1}
|q_{x_k,*} - q_{*,0}|^2&\leq \lim_{k\to \infty} \int_{\partial B_1}\biggl|\frac{u(x_k+r_\ep x)-p_{*,x_k}(r_\ep x)}{r_\ep^3} - q_{*,0}\circ R_k(x)\bigg|^2+Cr_\ep\\
& \le	 \ep+Cr_\ep.
\end{align*}
Since $\ep>0$ is arbitrary, this proves the continuity at $0$.
In addition, arguing as above (using Lemma \ref{lem:mon 3}) one sees that the convergence in \eqref{eq:3rd limit} is locally uniform with respect to $x_\circ$.
\end{proof}

\begin{remark}
It is important to observe that the above proof shows something stronger:
if $x_k\in \Sigma_m^g$ (so their frequency is at least $3$, see \eqref{eq:def anomalous}) and $x_k \to x_\circ$ with $x_\circ \in \Sigma_m^{3rd}$, then
\begin{equation}
\label{eq:cont k}
\lim_{k\to \infty}\int_{\partial B_1}
|q_{x_k} - q_{*,0}|^2=0
\end{equation}
whenever $q_{x_k}$ is an arbitrary limit point 
of $r^{-3}\left(u(x_k+r x)-p_{*,x_k}(r x)\right)$ as $r \to 0$. In other words, even if the third order blow-up of $u-p_{*,x_k}$ at $x_k$ may not be unique, any such limit has to converge to $q_{*,0}$ as $x_k\to x_\circ$.
 Indeed, if $\{r_{k,j}\}_{j\geq 1}$ is a sequence converging to $0$ such that
$$
q_{x_k}(x)=\lim_{j\to \infty}\frac{u(x_k+r_{k,j} x)-p_{*,x_k}(r_{k,j} x)}{r_{k,j}^3},
$$
then Lemma \ref{lem:mon 3} applied at $x_k$ with $q=q_{*,0}\circ R_k$ yields (since $r_{k,j}\leq r_\ep$ for $j\gg 1$)
\begin{align*}
\int_{\partial B_1}|q_{x_k} - q_{*,0}\circ R_k|^2 &=\lim_{j \to \infty}
 \int_{\partial B_1}\biggl|\frac{u(x_k+r_{k,j} x)-p_{*,x_k}(r_{k,j} x)}{r_{k,j}^3} - q_{*,0}\circ R_k(x)\bigg|^2\\
 & \leq 
  \int_{\partial B_1}\biggl|\frac{u(x_k+r_\ep x)-p_{*,x_k}(r_\ep x)}{r_\ep^3} - q_{*,0}\circ R_k(x)\bigg|^2 +Cr_\ep,
\end{align*}
and the result follows as in the proof of Proposition \ref{prop:C2}.

Note also that, when $n=2$, $q_{*,0}$ is odd with respect to the line $\{p_{*,0}=0\}$ (see the proof of Lemma \ref{lemdim20}).
Hence it follows by \eqref{eq:cont k}
and the continuity of $p_{*,x_k}$
that
\begin{equation}
\label{eq:cont k2}
\lim_{k\to \infty}\int_{\partial B_1}
|q_{x_k}^{\rm odd} - q_{*,0}|^2+\int_{\partial B_1}
|q_{x_k}^{\rm even}|^2=0,
\end{equation}
where $q_{x_k}^{\rm odd}$ (resp. $q_{x_k}^{\rm even}$)  is the odd (resp. even) part of $q_{x_k}$ with respect to $\{p_{*,x_k}=0\}$. In particular, as in Lemma \ref{lemdim20}, $q_{x_k}^{\rm odd}$ is a 3-homogeneous harmonic polynomial.
\end{remark}

As consequence of the previous results, we obtain the following result about the structure of $\Sigma_{n-1}$.
\begin{theorem}
\label{thm:C2}
The following holds:
\begin{enumerate}
\item[($n=2$)] $\Sigma_1$ is locally  contained in a $C^{2}$ curve.
\item[($n\geq 3$)] For any $m=1,\ldots,n-1$, the set $\Sigma_{m}$ can be covered by a countable family of $C^2$ $m$-dimensional manifolds, except for at most  a set of Hausdorff dimension $m-1$.
\end{enumerate}
\end{theorem}

\begin{proof}
We start with the case $n=2$.
Let us consider the map
\begin{equation}
\label{eq:cont q odd}
\Sigma_1\ni x_\circ \mapsto q_{x_\circ}^{\rm odd},
\end{equation}
where $q_{x_\circ}^{\rm odd}$ is the odd part of $q_{x_\circ}$ with respect to $\{p_{*,x_\circ}=0\}$, and:\\
- if $x_\circ \in \Sigma_1^{3rd}$, then $q_{x_\circ}=q_{*,x_\circ}$ is the third order limit provided by Proposition \ref{prop:C2};\\
- if $x_\circ \in \Sigma_1\setminus\Sigma_1^{3rd}$, then $q_{x_\circ}$ is an arbitrary limit point 
of $r^{-3}\left(u(x_\circ+r x)-p_{*,x_\circ}(r x)\right)$ as $r\to 0$ (recall that $\Sigma_1=\Sigma_1^g$ for $n=2$, see Lemma \ref{possiblefreq}).

Fix $R\in (0,1)$, and 
for $(r,x_\circ) \in (0,1-R] \times (\Sigma^1\cap \overline B_R) $, let us define the function
\[
\mathcal F(r,x_\circ) :=  
r^{-3} \biggl(\ave_{B_r} \left(u(x_\circ+\cdot) -p_{*,x_\circ} - q_{x_\circ}^{\rm odd} \right)^2 \biggr)^{1/2 }.
\]
Note that, as a consequence of Lemma \ref{lem:mon 3}, the map $r\mapsto \mathcal F(r,x_\circ)$ is almost monotone, hence the limit as $r\to 0^+$ exists. Also, for $r>0$ fixed, the map $$\Sigma^1\cap \overline B_R\ni x_\circ\mapsto \mathcal F(r,x_\circ)$$ is continuous as a consequence of \eqref{eq:cont k} and \eqref{eq:cont k2} (recall that the set $\Sigma_1\setminus \Sigma_1^{3rd}$ consists of isolated points by Lemma \ref{lem:sigma h}(i), so $\mathcal F(r,\cdot)$
is trivially continuous at such points).
Thus, as in Lemma \ref{lem:laststep},
the almost monotonicity implies that $x_\circ\mapsto \mathcal F(0^+,x_\circ)$ is upper semicontinuous. Since $\mathcal F(0^+,\cdot)=0$
on $\Sigma_1^{3rd}$ (by Proposition \ref{prop:C2}) we deduce that, for any $\ep>0$, there exists $r_\ep>0$ such that
\begin{equation}
\label{eq:cont F}
\mathcal F(r,x_\circ)\leq \ep \qquad \forall\,x_\circ \in \Sigma_1 \cap \overline B_R
\quad \text{s.t.}\quad {\rm dist}(x_\circ,\Sigma_1^{3rd})\leq r_\ep,\qquad \forall\,r\in (0,r_\ep].
\end{equation}
Now, to any point $x_\circ\in \Sigma_1$ we associate the third order polynomial 
\[P_{x_\circ} (x) := p_{*,x_\circ}(x-x_\circ)+q_{x_\circ}^{\rm odd}(x-x_\circ),\]
and we consider the function $\mathcal  G:\Sigma_1 \times \Sigma_1\to \R$ defined as 
$$
\mathcal  G(x_\circ, x):= \frac{1}{\rho_{x_\circ, x}^3}\big\| (P_{x_\circ} -P_x)(\rho_{x_\circ, x}\,\cdot\, )\big\|_{L^2(B_1)} ,\qquad \rho_{x_\circ, x}:=|x-x_\circ|.
$$
We want to prove that $\mathcal G$ is uniformly continuous on $(\Sigma_1 \cap \overline B_R)\times (\Sigma_1 \cap \overline B_R)$ for any $R\in (0,1)$.

Observe that, thanks to Lemma \ref{lem:sigma h}(i), the set
$$
O_{r,R}:=\left\{x_\circ \in \Sigma_1 \cap \overline B_R\,:\,{\rm dist}(x_\circ,\Sigma_1^{3rd})\geq r\right\}
$$
is finite for any $r>0$. In particular, if we define
$$
U_{r,R}:=\left\{x_\circ \in \Sigma_1 \cap \overline B_R\,:{\rm dist}(x_\circ,\Sigma_1^{3rd})\leq r\right\},
$$
then for any $\ep>0$
there exists $\delta=\delta(\ep)>0$ small enough such that
$$
{\rm dist}(x_1,x_2)> \delta\qquad \forall\,
(x_1,x_2) \in (O_{\ep,r_\ep/2}\times O_{\ep,r_\ep/2})\cup (O_{\ep,r_\ep}\times U_{\ep,r_\ep/2})
$$
(here $r_\ep>0$ is as in \eqref{eq:cont F}). 
Hence it is enough to check the continuity of $\mathcal G$ on $U_{\ep,r_\ep}\times U_{\ep,r_\ep}$.

Note that, arguing exactly as in \eqref{triangle!}, it follows that
\[
\mathcal  G(x_\circ, x) \le \mathcal F(\rho_{x_\circ, x}, x_\circ) + \mathcal F(2\rho_{x_\circ, x}, x)\qquad \forall \,x_\circ, x \in \Sigma_1. 
\]
In particular, provided $\rho_{x_\circ, x}=|x-x_\circ|\leq r_\ep/2$, then it follows by \eqref{eq:cont F} that
$\mathcal  G(x_\circ, x) \leq 2\ep$
whenever $(x_\circ, x) \in U_{\ep,r_\ep}\times U_{\ep,r_\ep}$, which proves the desired uniform continuity of $\mathcal G$.

Since the norm $\|\cdot\|_{L^2(B_1)}$ is equivalent to the norm $\|\cdot\|_{C^3(B_1)}$  on the space of third order polynomials, 
the uniform continuity of $\mathcal G$
implies that the the polynomials $P_{x_\circ}$ are  continuous in the sense of Whitney's Theorem: for any $R \in (0,1)$ there exists a modulus of continuity $\omega_R$ such that
$$
| D^kP_{x_\circ} (x) -D^k P_x(x) | \le  \omega_{R}(|x-x_\circ|)|x-x_\circ|^{3-k} \qquad \forall\,x_\circ,x \in \Sigma_1 \cap \overline B_R,\,k=0,1,2,3.
$$
Since $\Sigma_1$ is closed, the set $\Sigma_1 \cap \overline B_R$ is compact, so
this allows us to apply the classical Whitney's Theorem to find a map $F \in C^3(\R^2)$ such that
$$
F(x)=p_{*,x_\circ}(x-x_\circ)+q_{*,x_\circ}(x-x_\circ)+o(|x-x_\circ|^3) \qquad \forall\, x_\circ \in \Sigma_1^{3rd}\cap \overline{B_R},
$$
and we conclude by the Implicit Function Theorem (see the proof of Lemma \ref{lem:laststep}).

\smallskip

Concerning the higher dimensional case,
since ${\rm dim}_\HH(\Sigma_{m} \setminus \Sigma_{m}^{3rd})\leq m-1$
(see Lemma \ref{lem:sigma h}),
for any $j \in \mathbb N$ we can find a countable family of balls $\{\hat B_i\}$ such that
$$
\Sigma_{m} \setminus \Sigma_{m}^{3rd}
\subset \bigcup_i \hat B_i=:\mathcal O_j,
\quad \text{and}\quad \sum_i{\rm diam}(\hat B_i)^{m-1+1/j}<\frac1j.
$$
In particular $\mathcal H^{m-1+1/j}_\infty(\mathcal O_j)<1/j$ (see \eqref{eq:def Haus}). Note that, because $\Sigma_{m}$ is relatively open in $\Sigma$  (by the continuity of $x_\circ\mapsto p_{*,x_\circ}$), the set $\overline {\Sigma_m}\setminus \Sigma_m$ closed. Define the sets  
$$
\mathcal U_j := \{x \, :\, {\rm dist}(x,\overline{\Sigma_m}\setminus\Sigma_m)<1/j\},\qquad K_j:=\Sigma_{m}\setminus (\mathcal O_j\cup \mathcal U_j).
$$
Since $\mathcal O_j$ is open, the set $K_j$ is closed. 
Noticing that the polynomials $P_{x_\circ} (x) := p_{*,x_\circ}(x-x_\circ)+q_{*,x_\circ}(x-x_\circ)$
are continuous with respect to $x_\circ \in K_j$ (by Proposition \ref{prop:C2}), we can argue as we did above in case $n=2$ to conclude that $K_j$ can be locally covered by a $m$-dimensional manifold of class $C^2$.
Then the result follows by observing that
$\cup_jK_j=\Sigma_m\setminus \left(\cap_j \mathcal O_j)\right)$ and 
 $\HH^{\beta}_\infty(\cap_j \mathcal O_j)=0$ for any $\beta>m-1$, hence
${\rm dim}_\HH(\cap_j \mathcal O_j)\leq m-1 $ (see \eqref{eq:def dim}).
\end{proof}

\appendix
\section{Examples of $(m-1)$-dimensional anomalous sets $\Sigma^a_m$}

In this Appendix, for any $1\leq m \leq n-2$, we construct an example of solution to the obstacle problem in $\R^n$ 
for which the anomalous set $\Sigma_m^a$
is $(m-1)$-dimensional. The existence of such examples shows that the assertion  ${\rm dim}_\HH(\Sigma^a_m) \le m-1$  in  Theorem \ref{thm:main}(b) is optimal.

These solutions are constructed as follows:
given $1\leq m \leq n-2$ we consider  functions 
\begin{equation}\label{formu}
u(x_1,x_2,  \dots, x_n) = u^\star(z,r), \qquad    z:= x_{m},  \quad  r:= \sqrt{x_{m+1}^2 +x_{m+2}^2+\cdots+ x_{n}^2},
\end{equation}
which are independent of the first $(m-1)$ variables and are axially symmetric with respect to the last $m-n$ variables. Then our goal is to find solutions to the obstacle problem which are of the form \eqref{formu} and for which all points in the $(m-1)$ dimensional affine space 
\[\mathcal Z := \{x_m= x_{m+1} =\cdots=x_n=0 \}\]
 are anomalous points in  $\Sigma_m^a$.
To build these examples, we rely on some ideas introduced in \cite{Y16}.

Fix $\phi:[-1,1] \rightarrow \R$ a nonnegative $C^2$ function satisfying 
\[
\phi(0)>0,\quad \phi(1)=0,
\quad \phi'(z)\le 0 \quad \forall\,z\in(0,1),\quad \text{and} \quad \phi(-z)=\phi(z).
\]
Then, for any real number $k>0$ we consider the solution $u^k\ge 0$ to the obstacle problem in $(-1,1)\times(0,1)\subset \R^2$
\[
\begin{cases}
{\rm div}( r^{n-m-1}\nabla u^k) = k\,r^{n-m-1}\,\chi_{\{u^k>0\}}   \quad&\mbox{in } (-1,1)\times(0,1),
\\
u^k(\pm1,r) =0&  r\in (0,1),
\\ 
u^k(z,1) = \phi(z) &  z\in (-1,1).
\end{cases}
\] 
In other words, $u^k(x_{m}, \sqrt{x_{m+1}^2+\cdots+ x_{n}^2})$ is a solution of the classical obstacle problem in the cylinder $\R^{m-1} \times (-1,1)\times B_1^{(n-m)}$
satisfying the symmetries in \eqref{formu}.

Note that, when we think of this obstacle problem as a two dimensional problem (in the variables $z,r$), we do not prescribe ``boundary'' values at $r=0$ since this line has zero capacity for the operator ${\rm div}( r^{n-m-1}\nabla \,\cdot\,)$ when $n-m-1\ge1$ (this is equivalent to saying that the set $\{x_{m+1} =\cdots=x_n=0 \}$ has zero harmonic capacity in $\R^n$).

We now claim that:
\begin{enumerate}
\item[(i)]  $u^k$ is even $z$, namely $u(z,r)=u(-z,r)$;
\item[(ii)] $\partial_z u \le 0$ in $(0,1)\times(0,1)$;
\item[(iii)] the set $\{u^k>0\}$ is convex in the direction $z$, and is symmetric with respect to $z=0$.
\end{enumerate}
 Indeed, (i) follows from the symmetry of the boundary data (and the uniqueness of solution to the obstacle problem).
 
 To show (ii) we consider the open set $U := (0,1)\times (0,1) \setminus \{u^k=0\}$ and observe that  $\partial_z u^k \le 0$ on $\partial U$. Indeed: 
\begin{itemize}
\item  $\partial_z u^k =0$ on $\{z=0\}$, by symmetry;
\item  $\partial_z u^k =0$ on $\partial  \{u^k=0\}$, since $u^k$ is nonnegative and of class $C^{1,1}$;
\item  $\partial_z u^k \le 0$ on $\{z=1\}$, since $u^k(r,1) =0$ while $u^k(r,z) \ge 0$ for $z<1$;
\item  $\partial_z u^k =\phi' \le 0$ on $\{r=1\}$. 
\end{itemize}
As a consequence, since ${\rm div}( r^{n-m-1}\nabla (\partial_z u^k))=0$ inside $U$ (this follows by differentiating the equation for $u$ with respect to $z$), we deduce  by the maximum principle that $\partial_z u^k \le 0$ in $U$.
Also, we note that $\partial_z u^k = 0$ in $(0,1)\times (0,1) \setminus U$ since $u^k\equiv 0$ there (recall that $u^k$ is nonnegative and of class $C^{1,1}$), proving (ii).

Finally, (iii) is an immediate consequence of (i) and (ii).

We now observe that, for $k$ sufficiently large, the contact set contains a neighborhood of the origin, and hence it must contain a cylindrical neighborhood of $\{r=0\}$ (thanks to (iii)). On the other hand, for $k\ll 1$ we have $u^k(0)>0$. 
Therefore, by continuity, there exists (a unique) $k_{\star}>0$ such that $\partial\{u^{k_\star}>0\}$ touches tangentially the line $\{r=0\}$. 

Set $u^\star := u^{k_\star}$. 
Observe that that, with this definition, the function 
\begin{equation}\label{formu2}
u(x):=u^{\star}(x_{m}, \sqrt{x_{m+1}^2 +\cdots+ x_{n}^2})
\end{equation}
has a full $(m-1)$-dimenional space of singular points on $\mathcal Z=\{x_m=x_{m+1} = \cdots=x_n=0\}$. 
Also, by the given symmetry, these singular points belong to the stratum $\Sigma_{m}$.

We now prove the following:
\begin{proposition}
Let $u$ be the  symmetric solution defined in \eqref{formu2}. Then the set $\mathcal Z\subset \Sigma_m$ consists of anomalous points, that is $\mathcal Z\subset \Sigma_m^a$.
\end{proposition}
\begin{proof}
The proof consists of three steps. 
\smallskip

- {\em Step 1}. We show that $u$ has no other singular points in a neighborhood of $\mathcal Z$ (except of course the points in $\mathcal Z$).

To prove this, assume by contradiction that there exists a sequence of singular points $x_\ell\rightarrow 0$. Note that, since $u$ is invariant in the first $m-1$ variable, we can assume that $x_\ell \in \{x_1=\ldots=x_{m-1}=0\}$.
Then, by the symmetries of $u$ and Lemma \ref{codge2}, a blow-up  $q$ of $u-p_*$ at $0$ is a homogeneous harmonic polynomial that vanishes on the $x_m$-axis and that enjoys the same symmetries as $u$ (note that $p_*$ has the same symmetries as $u$ and vanishes on $r=0$). Thus, $q=q^\star(z,r)$, where $q^\star$ solves
\[
{\rm div}( r^{n-m-1} \nabla q^\star)=0,   \qquad q^\star(z,0) = 0 \quad \forall\, z,  \qquad  \mbox{and}\quad q^\star\not\equiv 0.
\]
Let $m$ be the degree of $q$. Since $q$ is smooth and $q^\star(z,0) = 0$, $q^\star$ must be a polynomial of the form
$\sum_{1\leq k\leq m/2}a_kz^{m-2k}r^{2k}$.
Let $k_0\geq 1$ be the first index such that $a_{k_0}\neq 0$.
Then
$$
0={\rm div}( r^{n-m-1} \nabla q^\star)
=\Bigl(2k_0[n-m+2(k_0-1)]a_{k_0}z^{m-2k_0}+r^{2}Q(z,r)\Bigr)r^{n-m-1+2(k_0-1)},
$$
where $Q$ is a polynomial of degree $m-2k_0-2$.
In particular, the terms inside the parenthesis cannot be identically zero, giving the desired contradiction.

As a consequence, there is a neighborhood of $\mathcal Z$ which is free of singular points. (except the points in $\mathcal Z$).
In particular, as a consequence of \cite{C77,C98}, there exists $\ep>0$ such that $(\partial \{u^{\star}>0\} \setminus \{0\})\cap B_\ep$ is a smooth curve contained inside $(-\ep,\ep)\times (0,\ep)\setminus\{0\}$.

\smallskip

- {\em Step 2}. We show that, for any $\alpha>0$ small, there exists $R =R(\alpha)>0$ small such that the following holds:
\begin{equation}\label{goal2222}
\forall \,\varrho_\circ\in(0,R), \ \exists\, Y_{\varrho_\circ} \in  \partial \{u^\star >0\}\cap B_{2\varrho_\circ} \quad \mbox{ s.t. }\quad Y_{\varrho_\circ}\cdot \boldsymbol e_r \ge \varrho_\circ^{1+\alpha},
\end{equation}
where $Y_{\varrho_\circ}\cdot \boldsymbol e_r$ denotes the $r$-component of the point $Y_{\varrho_\circ}$.

Let $\rho := \sqrt{r^2+z^2}$ and  $\theta := \arctan (z/r) \in (-\frac \pi 2,\frac \pi 2)$ be polar coordinates in  ${(0,1)\times(-1,1)}$. 
We now state the following fact, whose proof is postponed to the end of the Step 2. 

\smallskip

\noindent {\bf Claim.} For any  $\alpha>0$ small, there exists $\delta = \delta(\alpha)>0$, and a function $\Theta_\alpha: [\delta,\frac \pi 2 -\delta] \rightarrow \R$ smooth, such that  $S_\alpha(r,z) := \rho^{1+\alpha/4} \Theta_\alpha(\theta)$ satisfies
\[
{\rm div}( r^{n-m-1} \nabla S_\alpha)=0 \quad \mbox{ in the cone $\mathcal C_\delta := \left\{(\rho,\theta)\in (0,1)\times (\delta,\frac \pi 2 -\delta)\right\},$}
\] 
and 
\begin{equation}
\label{eq:theta}
\Theta_\alpha(0)= \Theta_\alpha\Big(\frac \pi 2-\delta\Big) =0,\qquad 
 \Theta_\alpha>0\quad \text{for $\theta\in \Big(\delta,\frac \pi 2-\delta\Big)$}. 
\end{equation}

\smallskip

Note that, since $0$ is a singular point and recalling the symmetries of $u$, $u(x)=p_*(x)+o(|x|^2)$ where $p_*(x)=\frac{1}{2(n-m)}\sum_{i=m+1}^nx_i^2$.
This implies that $u^\star(z,r) \geq \frac{1}{4(n-m)}r^2+o(z^2)$ near the origin. Thus,
given $\alpha>0$, there exists $\eta=\eta(\alpha)>0$ small enough such that the inclusion
\begin{equation}
\label{eq:inclusion}
\mathcal C_\delta \cap B_\eta \subset\mathcal C_{\delta/2} \cap B_\eta\subset  \{u^\star >0\}
\end{equation}
holds. Note that, up to reducing $\eta$, we can assume that $\eta<\ep,$ so that (by Step 1) $\partial\{u^\star>0\}\setminus \{0\}$ is a smooth curve inside $B_\eta$.

Let  us consider the function 
$h :=-\partial_z u^{\star}$, and recall that $h\geq 0$ (by property (iii) in the construction of $u^\star$). Also, differentiating the equation 
$$
{\rm div}( r^{n-m-1}\nabla u^\star) = k_\star\,r^{n-m-1}\qquad \text{in $\{u^\star>0\}$}
$$
with respect to $z$,
we obtain
\begin{equation}
\label{eq:h}
{\rm div}( r^{n-m-1}\nabla h) = 0\qquad \text{in $\{u^\star>0\}$}.
\end{equation} 
Since $h>0$ inside $\{u^\star>0\}$ (by the strong maximum principle), it follows that
\begin{equation}
\label{eq:free h}
\partial\{h >0\}\cap B_\eta\cap \{z>0\}=\partial\{u^\star >0\}\cap B_\eta\cap \{z>0\}
\end{equation} 
and $\mathcal C_\delta \cap \partial B_\eta\subset \subset \{u^\star>0\}$. Thus
there exists a constant $c_0=c_0(\alpha)>0$ such that $h\geq c_1S_\alpha$ on $\mathcal C_\delta \cap \partial B_\eta$. In addition $h\geq 0=S_\alpha$ on $\partial\mathcal C_\delta \cap B_\eta$. Hence, since $h$ and $S_\alpha$ solve the same equation,
it follows by the maximum principle that $h\geq S_\alpha$ inside $\mathcal C_\delta\cap B_\eta$. Recalling \eqref{eq:theta}, 
this implies that 
\begin{equation}\label{bybelow}
h \ge c_1\rho^{1+\alpha/4} \qquad \mbox{in }  \left\{ (\rho,\theta)\in (0,\eta)\times [2\delta ,\textstyle\frac \pi 2 -2\delta]\right\} 
\end{equation}
for some $c_1= c_1(\alpha)>0$.

We now want to find a lower bound on the normal derivative of $h$ at points on $\partial\{u^\star>0\}$. For this
we use a Hopf-type argument, constructing suitable barriers for our operator ${\rm div}(r^{n-m-1}\nabla \,\cdot\,)$.
These are given by the family of functions
\[S^{H}_{b}(r,z) := 2(r-b)  - (n-m-1)(z-1)^2,\quad\mbox{ where $b\in \R$}.\]
Note that 
\[
{\rm div}(  r^{n-m-1} \nabla S^{H}_b)=0\quad \mbox{in }\{r>0\},
\]
and $S^{H}_b<0$ outside of the parabolic region $\mathcal P_b:=\left\{b +\frac{n-m-1}{2} (z-1)^2 \le r\right\}$.

Now, given $0<\varrho_\circ\ll\eta$ we consider the rescaled function
\[ h_{\varrho_\circ} := \varrho_\circ^{-1-\alpha/2} h(\varrho_\circ\,\cdot\,)\]
and we note that (thanks to \eqref{bybelow})
\begin{equation}\label{bybelow2}
h_{\varrho_\circ}\bigr(\rho,\textstyle\frac \pi 2 -2\delta\bigl)\geq c_1\varrho_\circ^{-\alpha/4}\rho^{1+\alpha/4}\qquad \forall\,\rho \in (0,\textstyle\frac\eta{\varrho_\circ}).
\end{equation}

Consider now our barriers $S_b^H$. For $b>1/2$
the parabolic region $\mathcal P_b$ is contained inside the set $\left\{\theta<\textstyle\frac \pi 2 -2\delta\right\}$. In particular $S_b^H<0$ inside the cone $\hat{\mathcal C}_{2\delta} := \left\{\theta \in (\frac \pi 2-2\delta,\frac \pi 2) \right\}$. 
Then, we start decreasing $b$ until the first value $b_\circ$ such that $\partial\mathcal  P_{b_\circ}$ touches $\partial\{h_{\varrho_\circ} >0\}$.
Note that, thanks to \eqref{eq:free h} and Step 1,  $b_\circ >0$ and the contact point will happen for some $\tilde Y_{\varrho_\circ}\in \{r>0\}$.

Since
$h_{\varrho_\circ}\geq 0=S_{b_\circ}^H$ on $\partial\mathcal  P_{b_\circ}\cap \hat{\mathcal C}_{2\delta}$ and
$h_{\varrho_\circ}\geq c_2(\alpha)\varrho_\circ^{-\alpha/4}\geq S_{b_\circ}^H$ on $\mathcal P_{b_\circ}\cap \partial \hat{\mathcal C}_{2\delta}$ for $\varrho_\circ$ sufficiently small (see \eqref{bybelow2}), it follows by
the maximum principle that
 $h_{\varrho_\circ}\geq S_{b_\circ}^H$ inside $\mathcal P_{b_\circ}\cap \hat{\mathcal C}_{2\delta}$.
 Hence, since both $h_{\varrho_\circ}$ and $S_{b_\circ}^H$ vanish at $\tilde Y_{\varrho_\circ}$,
 we deduce 
that
\begin{equation}\label{normder}
\partial_\nu h_{\varrho_\circ} (\tilde Y_{\varrho_\circ}) \ge \partial_\nu S_{b_\circ}^H(\tilde Y_{\varrho_\circ}) \geq c_2>0, 
\end{equation}
where $\nu$ is the unit inwards normal to $\partial\{h_{\varrho_\circ} >0\}$, and $c_2=c_2(\alpha)$ is independent of $\varrho_\circ$ (here we use $b_\circ\le 1/2$). Observe also that, provided $\delta$ is small enough, $\mathcal P_{b_\circ}\cap \hat{\mathcal C}_{2\delta}\subset B_2$.

Rescaling \eqref{normder}, we obtain that for all $\varrho_\circ \in (0,\eta)$ small enough there is a point $Y_{\varrho_\circ} \in \partial \{u^\star >0\} \cap (0,1)\times (0,1)$ such that,
\begin{equation}\label{notc1alpha}
|Y_{\varrho_\circ}|\le 2\varrho_\circ\qquad \mbox{and} \qquad \partial_\nu  h(Y_{\varrho_\circ}) \ge c_2  \varrho_\circ^{\alpha/2},
\end{equation}
where $\nu$ denotes the unit inwards normal to $\{u^\star >0\}$ (see \eqref{eq:free h}).

To conclude the proof we observe that differentiating the equation 
\[{\rm div}(\nabla u^\star) = k^\star \, r\, \chi_{u^\star>0}\] with respect to $z$ we obtain $-{\rm div}(r^{n-m-1}\nabla h) =  k^\star\, r^{n-m-1} \nu_z\, \mathcal H^1|_{\partial\{u^\star>0\}}$, thus
\[
-r^{n-m-1} \partial_\nu h = k^\star\, r^{n-m-1} \nu_z \quad \mbox{on }{\partial\{u^\star>0\}}
\]
where $\nu_z=\nu \cdot \boldsymbol e_z$ denotes the $z$-component of  $\nu$.
Recalling \eqref{notc1alpha}, this proves that
\begin{equation}\label{goal1111}
\nu_z( Y_{\varrho_\circ}) \le  -\frac{c_2}{k^\star} \varrho_\circ^{\alpha/4}\qquad \mbox{and} \qquad |Y_{\varrho_\circ}| \le 2| \varrho_\circ|.
\end{equation}

Note that  \eqref{goal1111} already implies that $\partial \{u^\star >0\}$ cannot be a $C^{1,\alpha}$ curve at $0$.
We now prove the more precise estimate \eqref{goal2222}.

Since $b_\circ>0$ and the rescaled parabolic cap
$$\left\{(\varrho_\circ z,\varrho_\circ r)\,:\,{b_\circ} +\frac{n-m-1}{2} (z-1)^2 \le r \le 1 \right\}$$ touches $\partial \{u^\star =0\}$  at $Y_{\varrho_\circ}=(r_{\varrho_\circ},z_{\varrho_\circ})$, 
it follows that
\begin{equation}\label{eq:goal}
\frac{n-m-1}{2\varrho_\circ} (z_{\varrho_\circ}-\varrho_\circ)^2<\varrho_\circ{b_\circ} +\frac{n-m-1}{2\varrho_\circ} (z_{\varrho_\circ}-\varrho_\circ)^2= r_\circ 
\end{equation}
and that
\begin{equation}\label{eq:goal2}
\nu( Y_{\varrho_\circ})=\frac{\left(1,\frac{n-m-1}{\varrho_\circ}(z_{\varrho_\circ}-\varrho_\circ)\right)}{\sqrt{1+\left(\frac{n-m-1}{\varrho_\circ}(z_{\varrho_\circ}-\varrho_\circ)\right)^2}}.
\end{equation}
Combining \eqref{eq:goal} with \eqref{goal1111},
we get
$$
\frac{\frac{n-m-1}{\varrho_\circ}(z_{\varrho_\circ}-\varrho_\circ)}{\sqrt{1+\left(\frac{n-m-1}{\varrho_\circ}(z_{\varrho_\circ}-\varrho_\circ)\right)^2}} \leq -\frac{c_2}{k^\star} \varrho_\circ^{\alpha/4}
\qquad \Rightarrow\qquad \frac{n-m-1}{\varrho_\circ}(z_{\varrho_\circ}-\varrho_\circ)
\leq -\frac{c_2}{k^\star}\varrho_\circ^{\alpha/4}.
$$
Recalling \eqref{eq:goal}, this yields
$$
\frac{1}{2(n-m-1)}\Bigl(\frac{c_2}{k^\star}\Bigr)^2 \varrho_\circ^{1+\alpha/2}\leq r_\circ=Y_{\varrho_\circ}\cdot \boldsymbol e_r.
$$
Since $\varrho_\circ^{1+\alpha}\ll\varrho_\circ^{1+\alpha/2}$, this
proves \eqref{goal2222}.

We conclude Step 2 proving the claim.

\vspace{3pt}
\noindent \emph{Proof of Claim}. The linear function $z = \rho \sin \theta$ satisfies ${\rm div}( r^{n-m-1}  \nabla z)=0$.
Also $\Theta_0(\theta) = \sin \theta$ is positive on $(0,\pi/2)$ and satisfies $\Theta_0(0)=0$, hence $\Theta_0$ must be the minimizer of  the Rayleigh quotient 
\[
\min\bigg\{ \frac{ \int_0^{\pi/2} (\cos\theta)^{n-m-1} (\Theta')^2}{\int_0^{\pi/2} (\cos\theta)^{n-m-1} (\Theta)^2}  \ :\  \Theta(0)=0 \bigg\} ={n-m}
\]
(equivalently, recalling \eqref{formu}, the restriction to the harmonic function $x_m$ to $\R^{m-1}\times \mathbb S^{n-m}\subset \R^{m-1}\times \R^{n-m+1}$
is a minimizer of the Rayleigh quotient for the classical Dirichlet integral on $\mathbb S^{n-m}\cap \{x_m\geq 0\}$).

Since $n-m-1\ge1$ (this is where we crucially use this assumption), we have 
\[
\lim_{\delta\downarrow 0} \min\bigg\{ \frac{ \int_0^{\pi/2-\delta} (\cos\theta)^{n-m-1} (\Theta')^2}{\int_0^{\pi/2-\delta} (\cos\theta)^{n-m-1} (\Theta)^2}  \ :\  \Theta(\delta)=\Theta(\pi/2-\delta)=0 \bigg\}\downarrow n-m
\]
(this is equivalent to saying that the north pole in $\mathbb S^{n-m}\cap \{x_m\geq 0\}$ has harmonic capacity $0$, so the boundary condition $\Theta(\pi/2-\delta)=0$ disappears in the limit $\delta \to 0$).
Thus, by continuity, for any $\alpha>0$ small there exists $\delta=\delta(\alpha)>0$ small enough such that the previous Rayleigh quotient in $(\delta ,\pi/2-\delta)$ will give the value $$\mu_\alpha:=(n-m+\alpha/4)(1+\alpha/4).$$
This implies that if we denote by $\Theta_\alpha$ the first eigenfunction (i.e., the function attaining the minimal quotient value $\mu_\alpha$), then
$-( (\cos\theta)^{n-m-1}\Theta_\alpha')'=\mu_\alpha (\cos\theta)^{n-m-1}\Theta_\alpha$, and it follows by a direct computation (or by classical spectral theory) that $S_\alpha(r,z) := \rho^{1+\alpha/4} \Theta_\alpha(\theta)$ satisfies
${\rm div}( r^{n-m-1} \nabla S_\alpha)=0$, as desired. Finally, the strict positivity of $\Theta_\alpha$ inside $\left(\delta,\frac\pi2-\delta\right)$ is a classical property of the first eigenfunction.

\smallskip

- {\em Step 3.} We conclude the proof of the proposition by showing that  \eqref{goal2222} is incompatible with $0\in \Sigma^g_m$.
Recall that, by definition, $0$ belongs to $\Sigma^g_m$  (resp. $\Sigma^a_m$)  if $\phi(0^+,u-p_*) \ge3$   (resp. $\phi(0^+,u-p_*)<3$), see \eqref{eq:def anomalous}. 

Assume by contradiction that $\phi(0^+,u-p_*) \ge3$. Then by Lemma \ref{Hincreasing} we have
\[
\ave_{B_R} |u-p_*| \le \ \left(\ave_{B_R} (u-p_*)^2\right)^{1/2} \le CR^{3}\qquad \forall\,R>0.
\]
We now note that 
\[ \Delta(u-p_*) = -\chi_{\{u=0\}} \le 0, \]
so it follows by the mean value formula for superharmonic functions that
\begin{equation}\label{1234}
u-p_* \ge -CR^3 \quad \mbox{in } B_R.
\end{equation}
Recalling that $p_* = p_*(z,r) = \frac{1}{2(n-m)} r^2$ in the variables $(z,r)$, it follows by \eqref{1234} that
\[
-\frac{1}{n-m} r^2 \ge -C(r^2+ z^2)^{3/2}\quad \mbox{on } \partial\{u=0\}, 
\] 
therefore
\[
r \le C |(z,r)|^{3/2}\quad \mbox{on } \partial\{u =0\},
\] 
which clearly contradicts \eqref{goal2222} if we choose $\alpha<1/2$.
As a consequence $0$ (and  by symmetry all points on $\mathcal Z$) must belong to $\Sigma^a_m$.
\end{proof}

\end{document}